\documentclass[smallextended]{svjour3}
\usepackage{amsmath}
\usepackage{amssymb}
\usepackage{tikz}
\usepackage{mathptmx}
\usepackage{pgfkeys}
\usepackage{graphics}

\spnewtheorem*{spthm}{Main Theorem}{\bf}{\it}
\spnewtheorem{cor}{Corollary}{\bf}{\it}
\spnewtheorem*{sikthm}{Sikorav's existence theorem}{\bf}{\it} 
\spnewtheorem*{vitthm}{Viterbo's uniqueness theorem}{\bf}{\it} 
\spnewtheorem*{Shyp}{Standing hypotheses}{\bf}{\it}

\newenvironment{defi}{\medskip\noindent{\sc
Definition}. }{\goodbreak\medskip}
\newenvironment{defis}{\medskip\noindent{\sc
Definitions}. }{\goodbreak\medskip}
\newenvironment{nota}{\medskip\noindent{\sc
Notation}.}{\goodbreak\medskip}
\newenvironment{remk}{\noindent{\sc
Remark}. }{\goodbreak\vskip10pt}
\newenvironment{remks}{\noindent{\sc
Remarks}. }{\goodbreak\vskip10pt}
\newenvironment{notas}{\medskip\noindent{\sc
Notations}. }{\goodbreak\medskip}
\newenvironment{exas}{\noindent{\sc
Examples}. }{\goodbreak\vskip10pt}
\newenvironment{exa}{\noindent{\sc
Example}. }{\goodbreak\vskip10pt}

\def\cs{{\mathcal S}}

\def\cw{{\mathcal W}}
\def\ce{{\mathcal E}}

\def\cg{{\mathcal G}}

\def\ch{{\mathcal H}}
\def\ck{{\mathcal K}}
\def\cl{{\mathcal L}}
\def\cm{{\mathcal M}}
\def\cw{{\mathcal W}}
\def\cy{{\mathcal Y}}
\def\cu{{\mathcal U}}
\def\cv{{\mathcal V}}
\def\cz{{\mathcal Z}}
\def\ugoth{{\mathfrak u}}
\def\ug{{\mathfrak u}}

\def\R{\mathbb{R}}

\def\Z{\mathbb{Z}}
\def\N{\mathbb{N}}

\def\T{\mathbb{T}}

\def\d{\delta}

\def\x{\xi}
\def\y{\eta}

\def\dim{{\rm dim}\,}

\def\smallskip{\par\vspace{1mm}}
\def\medskip{\par\vspace{2mm}}
\def\bigskip{\par\vspace{3mm}}

\newcount\equationcount
\def\thenumber{0}
\def\eq#1{\global\advance\equationcount by 1
   \def\thenumber{\number\equationcount}
                        {$$#1\eqno(\thenumber)$$}}

\tikzset{
xmin/.store in=\xmin, xmin/.default=-1.5, xmin=-1.5,
xmax/.store in=\xmax, xmax/.default=7.5, xmax=7.55,
ymin/.store in=\ymin, ymin/.default=-0.75, ymin=-0.75,
ymax/.store in=\ymax, ymax/.default=3.25, ymax=3.25,
}

\newcommand{\axes} {
\draw[->] (\xmin,0) -- (\xmax,0);
\draw[->] (0,\ymin) -- (0,\ymax);
}

\begin{document}

\title{A multidimensional Birkhoff Theorem 
  for   time-dependent Tonelli Hamiltonians
  \thanks{The authors are supported by ANR-12-BLAN-WKBHJ and by
    MATH-AmSud project SIDIHAM}
}

\titlerunning{Birkhoff Theorem}
\author{Marie-Claude Arnaud$^{\dag}$
\thanks{$\dag$ member of the {\sl Institut universitaire de France.}}  
\and
Andrea Venturelli}


\institute{M-.C. Arnaud \at Avignon Universit\'e,
  Laboratoire de Math\'ematiques d'Avignon 
  (EA 2151)\\ F-84018  Avignon \\
  \email{marie-claude.arnaud@univ-avignon.fr}
  \and
  A. Venturelli \at Avignon Universit\'e,
  Laboratoire de Math\'ematiques d'Avignon 
  (EA 2151)\\ F-84018  Avignon \\
  \email{andrea.venturelli@univ-avignon.fr} }

\date{Received : April 2016 / Revised version : March 2017}

\maketitle
\date{}


\begin{abstract}  
Let $M$ be a closed and connected manifold,  
$H:T^*M\times {\mathbb R}/\Z\rightarrow \R$ a Tonelli $1$-periodic Hamiltonian and   
$\cl\subset T^*M$ a Lagrangian submanifold Hamiltonianly isotopic to the zero 
section. We prove that if $\cl$ is invariant by the time-one map of $H$, then 
$\cl$ is  a graph over $M$.\\
An interesting consequence in the autonomous case is that in this case, $\cl$ is invariant by all the time $t$ maps of the Hamiltonian flow of $H$.


\end{abstract}

\maketitle

\section{Introduction and Main Results.}\label{SecIntro}
A lot of problems coming from the physics are conservative, as the $N$-body problem and other classical mechanical systems: in other words, they are symplectic.\\
Close to the completely elliptic periodic orbits of symplectic dynamics, it is in general possible to use some change of coordinates called normal form (see \cite{Mos}) and thus to be led to study a local diffeomorphism
\begin{equation}\label{normalform}
(\theta, r)\in \T^n\times \R^n\mapsto (\theta+\alpha+\beta.r, r)+{\rm small}
\end{equation}
close to the zero section $\T^n\times \{ 0\}$ with $\beta$ being a symmetric matrix. When $\beta$ is a definite matrix, these diffeomorphisms are called twist maps and it can also be proved  that they are the time 1 map of a so-called Tonelli Hamiltonian vector field.\\
This kind of diffeomorphisms was introduced for example by Poincar\'e in the study of the circular restricted 3-body problem. When $n=1$, they were intensively studied by G.D.~Birkhoff.  
In \cite{Bir1}, G.D. Birkhoff proved that if $\gamma$ is an embedded circle of $\T\times \R$ that is not homotopic to a point and that is invariant by some conservative twist map, then $\gamma$ is the graph of a Lipschitz map $\T\rightarrow \R$.   A modern proof of this result can 
be find in \cite{He1}.
\medskip

{\bf Question:} what happens in higher dimensions?
\medskip

A natural extension of the 1-dimensional annulus $T^*\T=\T\times\R$ is the cotangent bundle $T^*M$ of a closed $n$-dimensional manifold $M$. We recall in section \ref{ssnota} that $T^*M$  can be endowed with a symplectic form.

If we want to obtain some submanifolds that are graphs (or more correctly sections) in $T^*M$, we are led to look at $n$-dimensional  submanifolds. Moreover, we have to impose some topological conditions for these submanifolds. Indeed, there are examples of conservative twist maps of $\T\times \R$ that have an invariant embedded circle that is homotopic to a point (and then this is not a graph): this happens for example for the time 1 map of the rigid pendulum close to the elliptic equilibrium.

But even if we ask that the invariant submanifold is homotopic to the zero-section of $T^*M$, it is easy to build examples of Tonelli  Dynamics that have an invariant submanifold that is not a graph but is homotopic to the zero section. The first author gave in  \cite{Arn1} an example of such a submanifold of $T^*\T^3=\T^3\times \R^3$ that is invariant by the Hamiltonian flow of $H(q,p)=\frac{1}{2}\| p\|^2=\frac{1}{2}( p_1^2+p_2^2+p_3^2)$, which is the geodesic flow for the flat metric on $\T^3$.\\

That is why we focus on the particular case of Lagrangian submanifolds. 

\begin{defi}
A submanifold $\cl\subset T^*M$ is {\sl Lagrangian} if 
$\dim \cl=n$ and $\omega_{|T\cl}=0$.

\end{defi}

Even if the set of Lagrangian submanifolds is very small in the set of all the $n$-dimensional submanifolds (more precisely it has no interior when endowed with the Hausdorff topology in $T(T^*M)$), there exist a lot of  invariant Lagrangian submanifolds   for the symplectic dynamics.

\begin{exas}
\begin{itemize}
\item  In $T^*\T$, a  loop  is always Lagrangian;
\item a vertical fiber $T^*_qM$ is Lagrangian;
\item the zero-section is Lagrangian;
\item more generally,  a $C^1$ graph is Lagrangian iff it is the graph of a   closed  1-form: for example,  $\{ (q, dS(q)); q\in M\}$ is a Lagrangian submanifold;
\item the stable or unstable (immersed) submanifold at a hyperbolic equilibrium is Lagrangian;
\item for the so-called completely integrable systems the phase space is foliated by invariant Lagrangian tori;
\item some of these invariant tori remain after perturbation (K.A.M. theory).
\end{itemize}
\end{exas}

Let us come back to the expression (\ref{normalform}). When $\beta$ is indefinite, M.~Herman constructed in
\cite{He2} some examples with an invariant Lagrangian torus that is isotopic to the zero section but not a graph. That is why we will assume that $\beta$ is positive definite, i.e. we will work with Tonelli Hamlitonians.

\begin{defi}
A $C^2$ function $H:T^*M\times \T\rightarrow \R$ is a {\em Tonelli Hamiltonian} if
\begin{itemize}
\item the Hamiltonian vector field\footnote{This will be defined in section \ref{ssnota}.} associated to $H$ is complete;  
\item $H$ is $C^2$-convex in the fiber direction, i.e. has a positive definite Hessian in the fiber direction;
\item $H$ is superlinear in the fiber direction, i.e. for every $B>0$, there exists $A>0$ such that:
$$\forall (q, p, t)\in T^*M\times \T, \| p\|\geq A\Rightarrow H(q, p, t)\geq B\| p\|.$$
\end{itemize}
\end{defi}

\begin{exa}
A Riemannian metric defines an autonomous Tonelli Hamiltonian.
\end{exa}

Let us recall that  we need to ask some topological condition on the invariant Lagrangian submanifold to be able to prove that it is a graph. To explain that, we need a definition.

\begin{defi}
Two submanifolds $\cl_1$ and $\cl_2$ of $T^*M$ are {\em Hamiltonianly isotopic} if there exists a time-dependent Hamiltonian $H:T^*M\times \R\rightarrow \R$ such that, if $(\phi_H^{s,t})$ is the family of symplectic maps that is generated by the Hamiltonian vector field of $H$, then
$$\cl_2=\phi_H^{0, 1}(\cl_1).$$
\end{defi}

Our main result is the following one.
\begin{spthm}
Let $M$ be a closed manifold, 
let $H:T^*M\times \T\rightarrow \R$ be a Tonelli $1$-time periodic 
Hamiltonian, and let $\cl\subset T^*M$ be a $C^1$ Lagrangian submanifold Hamiltonianly 
isotopic to a Lagrangian graph. If $\cl$ is invariant by the time 
one map associated to $H$,  then $\cl$ is the Lagrangian graph of a $C^1$ closed 1-form. 
\end{spthm}

A submanifold $\cl$ is Hamiltonianly isotopic to the zero-section if and only if the two submanifolds are isotopic in some particular subset $\ce$ of the set of Lagrangian submanifolds of $T^*M$, the set of the so-called {\sl exact Lagrangian submanifolds}.  Hence some related questions remain open.

{\bf Questions}
\begin{itemize}
\item Is the same conclusion true if we replace ``Hamiltonlianly isotopic'' by '' isotopic''?
\item Is the same conclusion true if we replace ``Hamiltonlianly isotopic'' by '' homotopic''?
\end{itemize}

Let us mention some related existing results.
\begin{itemize}
\item In \cite{Arn}, the first author proved that if a Lagrangian submanifold that is Hamiltonianly isotopic to a Lagrangian graph is invariant by a autonomous Tonelli Hamlitonian flow, then it has to be a graph. In next Corollary, we will explain how our result improves this statement.
\item In \cite {Be2}, P.~Bernard and J.~dos Santos extended this result in the autonomous case to the case of Lipschitz Lagrangian submanifolds.
\item In  \cite{BiaPol1} and \cite{BiaPol3}, for $M=\T^n$, with a weaker topological assumption (they assume that $\cl$ is homologous to the zero-section) but a strong hypothesis on the restricted Dynamics (it is assumed to be chain-recurrent), M.~Bialy and L.~Polterovich obtain the same result.
\item In \cite{He3},  M.~Herman proved a similar result for a submanifold that is:
\begin{itemize}
\item compact and Lagrangian;
\item with a Maslov class equal to $0$;
\item invariant by an exact symplectic twist map of $\T^n\times \R^n$ that is $C^1$-close enough to a completely integrable symplectic twist map;
\item such that the restricted dynamics is chain recurrent.
\end{itemize}
\end{itemize}

Our result is valid on all cotangent bundles and doesn't assume any dynamical behaviour of the restriction to the invariant submanifold.\\
However, it cannot be extended to any twist map because we don't know if a general twist map (in any dimension) is the time 1 map of a Tonelli Hamiltonian (see \cite{Gol} for an interesting discussion on this subject).\\
 
  Some arguments of our proof are common with the proof of the autonomous case in \cite{Arn}, but not all. Moreover, even if some of our techniques come from weak KAM theory, we tried to avoid to use the whole theory, as the Aubry set, the weak KAM solutions \dots and to write a self-contained article.

\begin{cor}\label{corauton}

Let $M$ be a closed manifold, 
let $H:T^*M\rightarrow \R$ be an autonomous  Tonelli  
Hamiltonian, and let $\cl\subset T^*M$ be a $C^1$ Lagrangian submanifold Hamiltonianly 
isotopic to a Lagrangian graph. If $\cl$ is invariant by the time 
one map associated to $H$,  then $\cl$ is the Lagrangian graph of a $C^1$ closed 1-form and  it is invariant by all the time $t$ maps $\phi_H^t$. 

\end{cor}
To any Tonelli Hamiltonian $H: T^*M\times\T\rightarrow\R$, a Lagrangian function $L:TM\times\T\rightarrow \R$ can be associated  via the Legendre duality.

$$\forall (q, v, t)\in TM\times \T,\  L(q, v, t)=\inf_{p\in T^*_qM}(p.v-H(q, p, t)).$$
\begin{defi}
A continuous and piecewise $C^1$ arc $\gamma_0: [a, b]\rightarrow M$ is {\em minimizing} if for every continuous and piecewise $C^1$ arc $\gamma: [a, b]\rightarrow M$ such that $\gamma_0(a)=\gamma(a)$ and $\gamma_0(b)=\gamma(b)$, we have
$$\int_a^bL(\gamma_0(t), \dot\gamma_0(t), t)dt\leq \int_a^bL(\gamma(t), \dot\gamma(t), t)dt.$$
\end{defi}

\begin{cor}\label{cormin}
Let $M$ be a closed connected manifold, 
$H:T^*M\times \T\rightarrow \R$ a Tonelli $1$-time periodic 
Hamiltonian, and $\cl\subset T^*M$ a $C^1$ Lagrangian submanifold Hamiltonianly 
isotopic to a  the zero section. If $\cl$ is invariant by the time 
one map associated to $H$,  then the orbit of every point of $\cl$ is minimizing. 
\end{cor}

\subsection{Notations}\label{ssnota}
 \begin{itemize}
\item $M$ is a closed Riemannian manifold, $\pi:T^*M\rightarrow M$ is its cotangent bundle and $\cz_{T^*M}$ the zero section; if $q=(q_1, \dots , q_n)$ are coordinates in a chart of $M$, the dual coordinates $p=(p_1, \dots, p_n)\in T_q^*M$ are defined by $p_i(\delta q_j)=\delta_{i,j}$ where $\delta q_i$ is the $i$th vector of the canonical basis and $\delta_{i, j}$ is the Kronecker symbol;
\item $T^*M$  is endowed with the Liouville 1-form that is defined by:
$$\forall p\in T^*M, \forall v\in T_p(T^*M), \lambda(v)=p\circ D\pi(p)(v);$$
in a dual chart, we have $\lambda=<p,dq>=\sum p_idq_i$;
\item the canonical symplectic form on $T^*M$ is $\omega=-d\lambda$; in a dual chart we have $\omega=dq\wedge dp=\sum dq_i\wedge dp_i$;
\item $\T=\R/\Z$ is the 1-dimensional torus with length 1 and $\T_2=\R/2\Z$ is  the 1-dimensional torus with length 2; we denote by $t\in\R\mapsto [t]_1\in \T$ and $t\in\R\mapsto [t]_2\in\T_2$ the corresponding covering maps;
\item a complete $C^2$ Hamiltonian $H:\T^*M\times\T\rightarrow \R$ being given, the Hamiltonian vector field $X_H$ is defined by $\omega (X_H(x,s), \delta x)=dH(x, s)\delta x$ and the corresponding Hamiltonian familly of diffeomorphisms is denoted by $(\varphi_H^{s,t})_{s, t\in \R}$.
\item we choose coordinates $(q, \tau)$ in the closed manifold $\cm_2=M\times \T_2$ and denote the dual coordinates by $(p, E)$; then the Liouville 1-form on $T^*\cm_2$ is $\Theta=<p,dq>+Ed\tau$ and the canonical symplectic form is $\Omega=-d\Theta=dq\wedge dp+d\tau\wedge dE$. We will often use the identification $T^*\cm_2=T^*M\times\T_2\times\R$;
\item we choose similarly coordinates $(q, \tau)$ in the closed manifold $\cm_1=M\times \T$ and denote the dual coordinates by $(p, E)$; then the Liouville 1-form on $T^*\cm_1$ is $\theta=<p,dq>+Ed\tau$. We will often use the identification $T^*\cm_1=T^*M\times\T\times\R$.
\item given a function $v : \cm_1 \rightarrow \R$
  (resp. $v: \cm_2 \rightarrow \R$) and a point
  $z=(q,t)\in \cm_1$ (resp. $z=(q,t)\in \cm_2$), if $v$ is differentiable on
  $z$ we set
  $$
  J_v(z)=(q,d_q v(z),t, \frac{\partial v}{\partial t}(z)).
  $$
  It is an element of $T^*M\times \T\times\R$ (resp. $T^*M\times \T_2\times\R$)
  and it can be identified with the differential of $v$ at $z=(q,t)$.
\end{itemize}
 
\subsection{A  useful reduction}

Let us explain why we will assume that $\cl$ is Hamiltonianly isotopic to the zero section (instead of ``to a Lagrangian graph'') in the proof.\\
 Assume that $\cl$ is Hamiltonianly isotopic to the Lagrangian graph $\cl_0$ and that $\cl$ is invariant by the time 1 map of the Tonelli Hamiltonian $H:T^*M\times \T\rightarrow \R$.\\
Then $\cl_0$ is the graph of some closed 1-form $\Lambda$ of $M$. Changing $\cl_0$ in a very close other graph, we can even assume that $\Lambda$ is smooth. Then $F: T^*M\rightarrow T^*M$ that is defined by $F(q, p)=(q, p+\Lambda(q))$ is a symplectic diffeomorphism so that
\begin{itemize}
\item $F^{-1}(\cl)$ is Hamiltonianly isotopic to the zero section $F^{-1}(\cl_0)=\cz_{T^*M}$;
\item $F^{-1}(\cl)$ is invariant by the time 1 map of the Tonelli Hamiltonian $\tilde H(q, p, t)= H( F(q, p), t)$.
\end{itemize}
Hence if we have proved the main theorem for the submanifolds that are Hamiltonianly isotopic to the zero section, we deduce that $F^{-1}(\cl)$ and then $\cl$ is a graph.

\subsection{Structure of the article}
\begin{itemize}
\item In the second section, we construct an extended autonomous Hamiltonian and an extended Lagrangian submanifold  in the extended phase space; we then build a graph selector for the extended Lagrangian submanifold;
\item in section 3, using the graph selector that was built in section 2, we build a dominated function;
\item in section 4, using the notion of calibrated curve, we prove the main theorem and its corollaries.
\end{itemize}

\section{Construction of a Lagrangian submanifold in the extended phase space and its graph selector.}
\label{extension}
\begin{Shyp}
{\rm From now, we assume that $H:T^*M\times \R/\Z\rightarrow \R$ is a Tonelli time $1$-periodic 
Hamiltonian and that $\cl\subset T^*M$ is a $C^1$ Lagrangian submanifold Hamiltonianly 
isotopic to $\cz_{T^*M}$ that is  invariant by the time 
one map associated to $H$.}
\end{Shyp}

The goal of this section is to build an extended autonomous Hamiltonian, an extended Lagrangian submanifold and a so-called graph selector. Moreover, we will prove some properties for these objects.

\subsection{Extension of the Lagrangian submanifold}
\vskip 1 mm
In this section, adding two dimensions to the phase space, we will replace the non-autonomous Hamiltonian flow by an autonomous one and extend the invariant submanifold in the new phase space. 

Let us comment on the choice of the new Hamiltonian and of the Lagrangian submanifold. The method that gives an autonomous Hamiltonian  is well-known but the extended Hamiltonian is not Tonelli with respect to the new variables and we cannot just apply the proof that the first author gave in \cite{Arn} in the autonomous case. The method to build an  extended  submanifold in the new phase-space is well-known too, but:
\begin{itemize}
\item a priori, this new submanifold has a boundary; thus we would need to build a theory of generating functions for manifolds with boundary to go on with our proof and we prefer to avoid this. Moreover, we don't know if this could work;
\item an idea to remove the problem of boundary is to identify what happens for the times $t=0$ and $t=1$. As the initial manifold is invariant, we can glue the two ends of the extended submanifold in a smooth way and obtain a closed manifold. Then a new problem appears: we cannot extend the isotopy that joins the zero section to the initial manifold in a periodic way because the submanifolds that appear in the isotopy are not invariant by the initial time 1 map and thus their extended submanifolds cannot be glued in a continuous way.
\end{itemize}
To overcome these problems, we had the idea to extend the Hamiltonian flow in a 2-periodic one by symmetrizing the extended Hamiltonian and   the extended submanifold. Let us explain this now.

We use the following function.
\begin{nota} {\rm
Let $\eta : \R\rightarrow \R$ be a non-negative and $C^\infty$ fonction satisfying
the following properties 
\begin{itemize}
\item[i)] $\eta(-t)=\eta(t)$ and $\eta(1-t)=\eta(1+t)$ for every $t\in\R$. 
\item[ii)] $\eta(0)=0$ and $\eta(1)=1$. 
\item[iii)] $\dot{\eta}(t)>0$ if $t\in (0,1)$.
\item[iv)] $\dot{\eta}(0)=\ddot{\eta}(0)=\dot{\eta}(1)=\ddot{\eta}(1)=0$.
\end{itemize}  

}\end{nota}
    
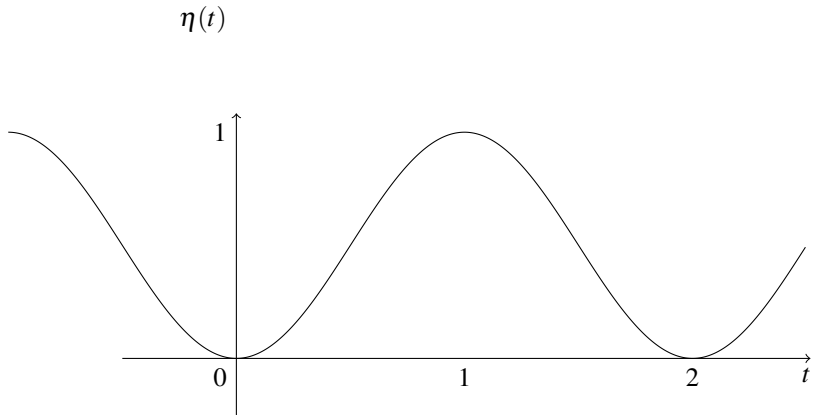
\vspace{3mm} 

\vspace{2mm} 


\begin{center}
\begin{tikzpicture} [samples=400]
\axes
\draw [scale=3] [domain=-1:2.5] plot(\x,{(1-cos(180*\x))/2});
\draw [scale=3] (2.5,0) node[below]{$t$} ;
\draw [scale=3] (0,1.5) node[left]{$\eta(t)$} ;
\draw [scale=3] (0,0) node[below left]{$0$} ;
\draw [scale=3] (1,0) node[below]{$1$} ;
\draw [scale=3] (2,0) node[below]{$2$} ;
\draw [scale=3] (0,1) node[left]{$1$} ;
\end{tikzpicture}
\end{center}

Let us introduce a new time-dependent Hamiltonian $K : T^*M\times \T_2\rightarrow \R$ defined 
by $K(q,p,t)=\dot{\eta}(t) H(q,p,\eta(t))$.
A straightforward  computation shows that 
every integral curve of $X_K$  can be written as 
$t\mapsto (q,p)(t)=(Q,P)(\eta(t))$,  where $s\mapsto (Q,P)(s)$ is an 
integral curve  of $X_H$. This fact can be expressed as follows :
$$
\varphi_K^{s,t}(q,p)=\varphi_H^{\eta(s),\eta(t)}(q,p).
$$
Roughly speaking, integral curves of $X_K$ are reparametrizations of segments of integral 
curves of $X_H$, but they slow down and turn back at integer time.  
In particular, integral curves of $X_K$ are all 
$2$-periodic and satisfy $(q,p)(-t)=(q,p)(t)$ and $(q,p)(1-t)=(q,p)(1+t)$.

Let now $\ck : T^*M\times \T_2\times\R\rightarrow \R$ be the autonomous 
Hamiltonian on the extended phase space 
$T^*\cm_2=T^*M\times\T_2\times\R$ defined 
by
$$
\ck(q,p,\tau,E)=K(q,p,\tau)+E.
$$ 
The Hamiltonian equations for  $\ck$ are given by 
\begin{equation} \label{champ_hamiltonien_ck}
\left\{
\begin{array}{rlrl}
\frac{dq}{dt}&=\frac{\partial K}{\partial p}(q,p,\tau), \qquad 
&\frac{d\tau}{dt}&=1, \\
\frac{dp}{dt}&=-\frac{\partial K}{\partial q}(q,p,\tau), \qquad
&\frac{dE}{dt}&=-\frac{\partial K}{\partial \tau}(q,p,\tau),
\end{array}
\right.
\end{equation}
We can see that the evolution on $T^*M$ is the same as for the Hamiltonian $K$, 
while the variable $\tau$ is essentially the time.  
If $\phi_\ck^t$ denotes the Hamiltonian flow of $\ck$, by 
(\ref{champ_hamiltonien_ck}), we 
easily get the relation 
\begin{equation} \label{flow_ck}
\phi_\ck^t(q,p,\tau,E)=(\varphi_K^{\tau,\tau+t}(q,p),\tau+t,
E+K(q,p,\tau)-K(\varphi_K^{\tau,\tau+t}(q,p),\tau+t)). 
\end{equation}
The evolution of the variable $E$ can be obtained using the conservation 
of $\ck$ by the flow.
Let $\cg\subset T^*M\times \T_2\times\R$ be the submanifold defined by 
$$
\cg=\{ \phi_\ck^t (q,p,0,-K(q,p,0)), (q,p)\in \cl,  \ t\in [0,2]\} .
$$
Since the flow $\phi_\ck^t$ is $2$-periodic in time, $\cg$ is a 
closed submanifold  diffeomorphic to $\cl\times \T_2$.
If we cut $\cg$ by a hypersurface $\tau=t$ and we forget the $E$ variable, 
we get the image of $\cl$ by $\varphi_K^{0,t}=\varphi_H^{0,\eta(t)}$. 
\begin{proposition}
If $\cl$ is Hamiltonianly isotopic to the zero section $\cz_{T^*M}$ of $T^*M$, then
$\cg$ is Hamiltonianly isotopic to the zero section $\cz_{T^*\cm_2}$ of $T^*\cm_2$. 
\end{proposition} 
\noindent 
\begin{proof} 
The proof is twofold. In the first part, we prove that $\cg$ is isotopic to some submanifold $\cg_0$   by using the fact that $\cl$ and  $\cz_{T^*M}$  are isotopic. In the second part, we prove that $\cg_0$ is isotopic to $\cz_{T^*\cm_2}$ by using the time-dependent Hamiltonian $s\mapsto sK$.

Let  $(\psi_s)$ be a Hamiltonian isotopy of $T^*M$ such that $\psi_0={\rm Id}_{T^*M}$ and $\psi_1(\cz_{T^*M})=\cl$. We use the notation $\cl_s=\psi_s(\cz_{T^*M})$. We denote by $h(q, p, s)$ a Hamiltonian associated to $(\psi_s)$.\\
For every $s\in[0,1]$, we define the submanifold $\cg_s$ of $T^*\cm_2=T^*M\times \T_2\times \R$ by
$$\cg_s=\{ \phi_\ck^t (q,p,0,-K(q,p,0)), (q,p)\in \cl_s,  \ t\in [0,2]\}.
$$
Exactly for the same reason as $\cg$, $\cg_s$ is a closed manifold that is diffeomorphic to $\cl\times \T_2$.  

Because $t\mapsto \varphi_K^{0, t}$ is 2-periodic, we can use this notation for $t\in \T_2$ too.\\
We define $F_s: T^*\cm_2\rightarrow T^*\cm_2$ by 
$$F_s(q, p, \tau, E)=(\varphi_K^{0, \tau}\circ \psi_s\circ \varphi_K^{\tau, 0}(q, p), \tau, E+K(q, p, 0)-K(\varphi_K^{0, \tau}\circ \psi_s\circ \varphi_K^{\tau, 0}(q,p), \tau)).$$
Note that $F_0={\rm Id}_{T^*\cm_2}$ and that $(F_s)$ is the Hamiltonian isotopy associated to the Hamiltonian $(q, p, \tau, E)\mapsto h(\varphi_K^{\tau, 0}(q, p), s)$.

As $F_s(\cg_0)=\cg_s$ and $\cg_1=\cg$,  $\cg$ is Hamiltonianly isotopic to $\cg_0$.

Let us now prove that $\cg_0$ is Hamiltonianly isotopic to the zero section $\cz_{T^*\cm_2}$ of $T^*\cm_2$.
\begin{defi}
A diffeomorphism $G:T^*M\rightarrow T^*M$ is {\em exact symplectic} if $G^*\lambda-\lambda$ is exact as a 1-form.
\end{defi}
\begin{lemma}\label{L1}
Let $(G_s)$ be an isotopy of exact symplectic diffeomorphisms. Then it is a Hamiltonian isotopy.
\end{lemma}
\begin{proof}
If we denote the Liouville 1-form on $T^*\cm_2$ by $\Theta$, we have: $G_s^*\Theta-\Theta=dS_s$ and then if $X_s$ is the vector field associated to $(G_s)$, we have: $G_s^*(L_{X_s}\Theta)=d\dot S_s$, i.e. $G_s^*(i_{X_s}d\Theta)+G_s^*(d(i_{X_s}\Theta))=d\dot S_s$. We finally obtain $G_s^*(i_{X_s}\Omega)=d(G_s^*(i_{X_s}\Theta)-\dot S_s)$. Hence $(G_s)$ is Hamiltonian and the associated Hamiltonian is $i_{X_s}\Theta-\dot S_s\circ G_s^{-1}$.
\end{proof}
In order to use Lemma \ref{L1}, we define $G_s: T^*\cm_2\rightarrow T^*\cm_2$ by:
$$G_s(q, p,\tau, E)=(\varphi_{sK}^{0, \tau}(q, p), \tau, E+s(K(q, p, 0)-K(\varphi_{sK}^{0, \tau}(q, p), \tau))).$$
Then $G_0=Id_{ T^*\cm_2}$ and $G_1(\cz_{ T^*\cm_2})=\cg_0$. If we succeed in proving that every $G_s$ is exact symplectic, we can deduce that $\cg_0$ is Hamiltonianly isotopic to the zero section and hence that $\cg$ is isotopic to the zero section. We prove that only for $s=1$ (we can replace $K$ by $sK$). We use the notations $G=G_1$, $\varphi_\tau=\varphi_K^{0, \tau}$ and $K_\tau(q, p)=K(q, p, \tau)$.

As $(\varphi_\tau)$ is a Hamiltonian isotopy, every $\varphi_\tau$ is exact symplectic. We write:
$\varphi_\tau^*\lambda-\lambda=dS_\tau$. We saw in the proof of Lemma \ref{L1}  that
$d(i_{X_K}\lambda-K_\tau)=d(\dot S_\tau\circ \varphi_\tau^{-1})$. By adding to $S_\tau$
a function depending only on $\tau$, we can assume that 
$i_{X_K}\lambda-K_\tau=\dot S_\tau\circ \varphi_\tau^{-1}$.

Now we compute
$$G^*\Theta-\Theta=\varphi_\tau^*(\lambda)+\varphi_\tau^*(i_{X_K}\lambda)d\tau -\lambda+(K(q, p, 0)-K(\varphi_{\tau}(q, p), \tau))d\tau.$$
Note that $K(q, p, 0)=0$. We obtain then
$$G^*\Theta-\Theta=dS_\tau+\varphi_\tau^*(i_{X_K}\lambda)d\tau  -K(\varphi_{\tau}(q, p), \tau)d\tau,$$
i.e. $G^*\Theta-\Theta=dS_\tau+\dot S_\tau d\tau$, then $G$ is exact symplectic.
\end{proof}
\subsection{The generating function }\label{ssgenerating}

There is a classical way of quantifying the Lagrangian submanifolds of a cotangent bundle that are Hamiltonianly isotopic to the zero section.
This is done by using the so-called generating functions.

The facts that we recall here come from different articles; more precisely, the existence theorem can be found in  \cite{Sik} and \cite{Bru} and the uniqueness
theorem is proved in \cite{Vi1} and \cite{Th}.

\begin{defis}  \begin{itemize}
\item Let $p:E\rightarrow \cm_2$ be a finite-dimensional vector bundle over $\cm_2$. A $C^2$ function $S: E\rightarrow \R$ is a {\em generating function} if its differential $dS$ is transversal to the manifold 
$\cw=\{ \xi\in T_e^*E; e\in E\quad{\rm and}\quad \xi=0\quad{\rm on}\quad T_e(p^{-1}(p(e)))\}$.

\item Then the {\em critical locus} $\Sigma_S$ of $S$ is the set $\Sigma_S=dS^{-1}(\cw)$.

\item The map $i_S: \Sigma_S\rightarrow T^*\cm_2$ is defined by $i_S(e): T_{p(e)}^*\cm_2\rightarrow T^*\cm_2$, $i_S(e)\delta x=dS(e).\delta e$ where $\delta e\in T_eE$ is any vector so that $dp(e).\delta e=\delta x$. 
\item If $\cg$ is a Lagrangian submanifold of $T^*M$, $S:E\rightarrow \R$ {\em generates} $\cg$ if $i_S$ is a diffeomorphism from $\Sigma_S$ onto $\cg$.
\item When the bundle $E=\cm_2\times \R^k$ is trivial and there exists a non-degenerate quadratic form $Q:\R^k\rightarrow \R$ such that $S=Q$ outside a compact subset, we say that $S$ is {\em special}. The {\em index} of $S$ is then the index of $Q$
\end{itemize}
\end{defis}
\begin{sikthm}
Let $\cg$ be a Lagrangian submanifold of $T^*\cm_2$ that is Hamiltonianly isotopic to the zero section. Then $\cg$ admits a special generating function.
\end{sikthm}

\begin{nota}
We denote a special generating function of $\cg$ by $S(q, \tau; \xi)$ with $\xi\in \R^k$. 
\end{nota}
\begin{remks}
When $S:\cm_2\times \R^k\rightarrow \R$ is special, we have:
$$\Sigma_S=\{ (q, \xi)\in \cm_2\times \R^k; \frac{\partial S}{\partial \xi}(q, \xi)=0\};$$
and 
$$\forall (q, \xi)\in \Sigma_S, i_S(q, \xi)=(q, \frac{\partial S}{\partial q}(q, \xi)).$$
\end{remks}
 
Observe that the condition that $S$ is a generating function means that  the map $\frac{\partial S}{\partial \xi}$ is a submersion at every point of $\Sigma_S$.

\begin{proposition}\label{P2generating}
The special functions $S_0:(q, \xi)\in M\times \R^k\rightarrow  S(q, 0; \xi)$ and $S_1:(q, \xi)\rightarrow S(q, 1; \xi)$ generate the Lagrangian submanifold $\cl\subset T^*M$.
\end{proposition}
\begin{proof}
The only non trivial thing to be proved is that the functions $S_i$ are generating function. Then the fact that they are special and that they generate $\cl$ is straightforward.\\
We recall that $$
\cg=\{\phi_\ck^t (q,p,0,0), (q,p)\in \cl,  \ t\in [0,2]\} .
$$
Hence if $(q, p)\in\cl$, we have the equalities  $X_\ck(q, p, 0, 0)=(0, 0, 1,0)\in T_{(q, p, 0, 0)}\cg$ and $X_\ck(\varphi_\ck^{0, 1}(q, p), 1, 0)=(0, 0, 1,0)\in T_{\phi_\ck^1(q, p, 0, 0)}\cg$ because $\dot{\eta}(0)=\ddot{\eta}(0)=\dot{\eta}(1)=\ddot{\eta}(1)=0$.\\
Let us recall that $i_S(q, \tau, \xi)=(q, \frac{\partial S}{\partial q}(q, \tau, \xi), \tau, \frac{\partial S}{\partial \tau}(q, \tau, \xi))$. Then, for $j=0, 1$, we have
$(D(i_S)^{-1}(q, p, j, 0))(0, 0, 1, 0)=(0, 1, \delta \xi_j)  \in T_{i_S^{-1}(q, p, j, 0)}\Sigma_S$. \\
As the equation of $\Sigma_S$ is $\frac{\partial S}{\partial \xi}(q, \tau, \xi)=0$, we deduce that
 $$\frac{\partial^2 S}{\partial \tau\partial \xi}(i_S^{-1}(q, p, j))=-\frac{\partial^2 S}{\partial \xi^2}(i_S^{-1}(q, p, j))\delta\xi_j.$$
This equality implies that for every $(q, p)\in\cl$ and $j=0, 1$, we have
$${\rm Im}\,D\left( \frac{\partial S}{\partial \xi}\right)(i_S^{-1}(q, p,j,0))={\rm Im}\,D\left(
\frac{\partial S}{\partial \xi}\right)(i_S^{-1}(q, p,j,0))_{|\delta\tau=0}=\R^k,$$
i.e. that $S_j$ is a generating function.
\end{proof} 

In the next subsection, we will build what is called a graph selector and we will prove that it doesn't depend on the generating function that we choose. To do that, we need a uniqueness result for the generating functions that is due to C.~Viterbo. Let us explain this.

\begin{defis}
Let $p:E\rightarrow \cm_2$ be a finite dimensional vector bundle and let $S:E\rightarrow \R$ be a generating function. Let us define the {\em basic operations} on generating functions:
\begin{itemize}
\item {\em Translation.} If $c\in\R$, then $S'=S+c:E\rightarrow \R$.
\item {\em Diffeomorphism.} If $p':E\rightarrow \cm_2$ is another vector bundle and $F: E'\rightarrow E$ is a diffeomorphism such that $p\circ F=p'$, then $S'=S\circ F:E'\rightarrow \R$.
\item {\em Stabilization.} If $p':E'\rightarrow \cm_2$ is another finite dimensional vector bundle endowed with a function $Q':E'\rightarrow \R$ that is quadratic non-degenerate when restricted to the fibers of $p'$, then $S'=S\oplus Q': E\oplus E'\rightarrow \R$.
\end{itemize}
Then, two generating functions are {\em equivalent} if they can be made equal after a succession of basic operations.
\end{defis}
\begin{remk}
Observe that the basic operations that are given by the two first item are reversible in the following sense: if $S'$ is obtained from $S$ by such an operation, then $S$ is obtained from $S'$ by a similar operation. This is not the case for the third basic operation, for which we can only add variables. \\
That is why the definition of equivalence is a little subtle: $S$ is equivalent to $S'$ if there exists a third generating function $S''$ so that $S''$ can be deduced from $S$ by some basic operations and $S''$ can be deduced from $S'$ by some basic operations.
\end{remk}

\begin{vitthm}
Two special functions that generate the same Lagrangian submanifold are equivalent.
\end{vitthm}

\begin{remk}
The property of being special is not preserved by the basic operations.
\end{remk} 
 
\subsection{Graph selector}
Using the generating function $S$, we will construct a {\sl graph selector} 
$u:\cm_2\rightarrow \R$. Such a graph selector was introduced by M.~Chaperon in \cite{Chap} (see \cite{PPS} and \cite{Sib} too) by using the homology. Here we prefer to use the cohomological approach. We now explain this.  
\begin{notas}
Let $p:E\rightarrow \cm_2$ be a finite dimensional vector bundle. If $S: E\rightarrow \R$ is a  function that generates a Lagrangian submanifold, $q\in \cm_2$ and $a\in\R$ is a real number, we denote the sublevel with height $a$ at $q$ by $$S^a_q=\{e\in E;  \quad p(e)=q\quad{\rm and}\quad S(e)\leq a\}$$ and we use the notation $S_q=S_{|E_q}$.
\end{notas}
When $S$ is special with index $m$, there exists $N\geq 0$ such that all the critical values are in $(-N, N)$. Then the De Rham relative cohomology space with compact support
$H^*(E_q, S_q^{-N})$ is isomorphic to $\R$ for $*=m$ and trivial if $*\not=m$. We denote by $\alpha_q$ a closed $m$-form with compact support on $E_q$  such that $\alpha_{q|S_q^{-N}}=0$ and $0\not= [\alpha_q] \in H^m(E_q, S_q^{-N})$.

If $a\in (-N, N)$, we use the notation $i_a: (S_q^a, S_q^{-N})\rightarrow (E_q, S_q^{-N})$ for the inclusion and then $i_a^*:H^m(E_q, S_q^{-N})\rightarrow H^m(S_q^a, S_q^{-N})$. The {\em graph selector} $u:\cm_2\rightarrow \R$ is then defined by:
$$u(q)=\sup\{ a\in \R; [i_a^*\alpha_q]=0\}=\inf\{ a\in \R; [i_a^*\alpha_q]\not=0\}.$$

\begin{proposition}\label{unicselector}
Let $p:E\rightarrow \cm_2$ be a finite dimensional vector bundle, let $S:E\rightarrow \R$ be a special generating function with index $m$ and let $\sigma:E'\rightarrow \R$ be a generating function that is got from $S$ after a succession of basic operations. If there are exactly $k$ stabilizations among these basic operations (with indices $m_1, \dots, m_k$), the sum of all the indices is denoted by $\displaystyle{\ell=m+\sum_{j=1}^k  m_j}$.\\
Then, for $N$ positive large enough, $H^{\ell}(E'_q, \sigma_q^{-N})$ is isomorphic to $\R$; if $[{\rm A}_q]$ is one of its generator, we can define a graph selector by 
$$U(q)=\sup\{ a\in \R; [i_a^*{\rm A}_q]=0\}=\inf\{ a\in \R; [i_a^*{\rm A}_q]\not=0\}.$$
This graph selector is equal to the one associated to $S$ plus a constant.

\end{proposition}

\begin{proof}
If the basic operation that we use is  a translation or a diffeomorphism, the proposition is straighforward. The only non trivial case concerns stabilization. From now, we forget the translations and the constants and we can assume that we are in the following case.

Assume that $S:E\rightarrow \R$ is a generating function such that  after a fiber diffeomorphism $\psi: E\rightarrow E$, $S\circ \psi$ is non-degenerate quadratic in every fiber outside some compact subset, the quadratic form being denoted by $Q_0$ and having   index $m_0$. Then $H^{m_0}(E_q, S_q^{-N})$ is isomorphic to $\R$ and $H^{*}(E_q, S_q^{-N})=\{ 0\}$ if $*\not= m_0$. For such a function we can define a graph selector $u$ as before (even if this function is not special).\\
Assume that  $Q:F\rightarrow \R$  is  a non-degenerate quadratic form with index $m$ when restricted to the fibers of $p':F\rightarrow\cm_2$ and let us use the notation $\cs=(S\circ\psi)\oplus Q: E\oplus F\rightarrow \R$. \\

{\em Dimension of $H^*(E_q\oplus F_q,\cs_q^{-N-C})$}

Observe that $|\cs-Q_0\oplus Q|$ is bounded by some constant $C$. Hence we have:
$$\{ Q_0\oplus Q\leq a-C\}\subset \{ \cs\leq a\}\subset \{ Q_0\oplus Q\leq a+C\}.$$
Then we choose $N\geq 0$ such that all the critical values of $S$ (and then of $\cs$) are in $(-N, N)$. Then the inclusion maps induce the following homomorphisms
$$ H^*(E_q\oplus F_q, Q_0\oplus Q\leq -N)\begin{matrix} j_3^*\\\longrightarrow\\{}\\ \end{matrix} H^*(E_q\oplus F_q, \cs_q^{-N-C})\begin{matrix} j_2^*\\\longrightarrow\\{}\\ \end{matrix}$$
$$H^*(E_q\oplus F_q, Q_0\oplus Q\leq -N-2C)\begin{matrix} j_1^*\\\longrightarrow\\{}\\ \end{matrix} H^*(E_q\oplus F_q, \cs_q^{-N-3C}).$$
As the pairs $(E_q\oplus F_q, Q_0\oplus Q\leq -N)$ and $(E_q\oplus F_q, Q_0\oplus Q\leq -N-2C)$ are homotopically equivalent and as the pairs  $(E_q\oplus F_q, \cs_q^{-N-C})$ and $(E_q\oplus F_q, \cs_q^{-N-3C})$ are homotopically equivalent, the maps  $j_1^*\circ j_2^*$ and $j_2^*\circ j_3^*$ are  isomorphisms, and then $j_2^*$ is an isomorphism too. We deduce that $H^*(E_q\oplus F_q, \cs_q^{-N-C})$ is isomorphic to $\R$ if $*=m_0+m$ and $\{ 0\}$ if $*\not=m_0+m$.\\
The same is true if we replace $\cs$ by the function $\sigma=S\oplus Q$ that will be denoted by $\sigma$ from now.\\

{\em A first inequality between the two graph selectors}

Let $\varepsilon$ be  a positive number. We will prove that $U(q)\leq u(q)+\varepsilon =: a+\frac{\varepsilon}{2}$.\\
Let $\alpha$ be a closed $m_0$-form that vanishes on $S_q^{-N}$ and is such that 
$0\not=[i_a^*\alpha]\in H^{m_0}(S_q^a, S_q^{-N})$ and let $\beta$ be a closed  $m$-form that vanishes on $Q_q^ \varepsilon$  and such that $0\not=[i_{\frac{\varepsilon}{2}}^*\beta]\in H^{m}(Q_q^\frac{\varepsilon}{2}, Q_q^{-\varepsilon})$. We denote by $A$ (resp. $B$) a $m_0$-cycle of $S_q^{a}$ with boundary in $ S_q^{-N}$ (resp.  $m$-cycle of $Q_q^\frac{\varepsilon}{2}$ with boundary in $Q^{-\varepsilon}$) such that $\alpha(A)\not=0$ (resp. such that $\beta (B)\not=0$). We use the notation $\mu_A=\sup Q_{|A}$ and $\mu_B=\sup S_{|B}$. Using the gradient flow of $Q_0$ on $S_q^{-N}$ to push $A$ or the gradient flow of $Q$ on $Q^{- \varepsilon}$ to push $B$, we can asssume that $S_{|\partial A}\leq -\varepsilon -N-\mu_A$ and  $Q_{|\partial B}\leq -\varepsilon -N-\mu_B$; observe that this implies that $\partial (A\times B)=(\partial A\times B)\cup (A\times \partial B)\subset \sigma^{-\varepsilon-N}$. \\
Then the cup product $\alpha\vee \beta$ is a closed $(m+m_0)$-form that vanishes  in $(Q_q^{-\varepsilon}\times F_q)\cup (E_q\times S_q^{-N})$ and such that $(\alpha\vee\beta)(A\times B)\not=0$. As the set $(Q_q^{-\varepsilon}\times F_q)\cup (E_q\times S_q^{-N})$ contains $\sigma^{-\varepsilon-N}$ and as the support of $A\times B$ is in $S_q^a\times Q_q^\frac{\varepsilon}{2}\subset \sigma^{a+\frac{\varepsilon}{2}}$, we deduce that 
$0\not=[i_{a+\frac{\varepsilon}{2}}^*(\alpha\vee \beta)]\in H^{m+m_0}(\sigma^{a+\frac{\varepsilon}{2}}, \sigma^{-\varepsilon-N})$ and thus $U(q)\leq a+\frac{\varepsilon}{2}=u(q)+\varepsilon$. Hence we have $U(q)\leq u(q)$.\\

{\em The reverse  inequality between the two graph selectors}

Let us now prove that for $\varepsilon >0$, we have $U(q)\geq u(q)-\varepsilon=a-\frac{\varepsilon}{2}$. We use the notation $j: \sigma^{a-\frac{\varepsilon}{2}}_q\rightarrow E_q\oplus F_q$, $j_1: S^a_q\rightarrow E_q$ and $j_2: Q^{-\varepsilon}_q\rightarrow F_q$ for the inclusion maps. \\
As $H^*(S^{a}_q, S^{-N}_q)$ and $H^*(Q^{-\frac{\varepsilon}{2}}_q, Q^{-\varepsilon}_q)$ are trivial,  there exists a $(m_0-1)$-form $\alpha_1$ on $S^a_q$ such that $\alpha_{1|S\leq -N}=0$ and $j^*_1\alpha=d\alpha_1$ and a $(m-1)$-form $\beta_1$ on $Q^{-\frac{\varepsilon}{2}}$ such that $\beta_{1|Q\leq -\varepsilon}=0$ and $j^*_2\beta =d\beta_1$. 

 Observe that $(S_q^a\times Q_q^{-\varepsilon}
,  S_q^{-N}\times Q_q^{-\frac{\varepsilon}{2}})$ is an excisive couple (see \cite{Mas}), hence all the following cohomology spaces vanish because they can be expressed with  the trivial spaces $H^*(S_q^a, S_q^{-N})$  and $H^*(Q_q^{-\frac{\varepsilon}{2}}, Q^{-\varepsilon})$ 
 $$H^*\left(S_q^a\times  Q_q^{-\frac{\varepsilon}{2}}, \left( S_q^a\times Q_q^{-\varepsilon}\right) \cup \left( S_q^{-N}\times Q_q^{-\frac{\varepsilon}{2}}\right) \right)=\{ 0\}.
 $$
 We have $d(\alpha_1\vee\beta)=d((-1)^{m_0}\alpha\vee\beta_1)=\alpha\vee\beta$.
 We deduce that there exists a $(m_0+m-2)$-form $\mu$ on $S_q^a\times  Q^{-\frac{\varepsilon}{2}}_q$ that vanishes on $(S_q^a\times Q_q^{-\varepsilon})\cup (S_q^{-N}\times Q_q^{-\frac{\varepsilon}{2}})$ and is such that $\alpha_1\vee\beta-(-1)^{m_0}\alpha\vee\beta_1=d\mu$. We can extend $\mu$ in a $(m_0+m-2)$-form that is defined on $E_q\oplus F_q$ and vanishes on $(E_q\times Q^{-\varepsilon}_q)\cup (S^{-N}_q\times F_q)$. Then the  $(m_0+m-1)$-form  $\alpha_1\vee\beta$ that is defined on $S_q^a\times F_q$ coincides  on the intersection of the two sets with the $(m_0+m-1)$-form $(-1)^{m_0}\alpha\times\beta_1+d\mu$ that is defined on $E_q\times Q^{-\frac{\varepsilon}{2}}$. Together, they define a $(m_0+m-1)$-form $\mu_1$ on $(S_q^a\times F_q)\cup (E_q\times Q^{-\frac{\varepsilon}{2}})\supset \sigma_q^{a-\frac{\varepsilon}{2}}$ such that
 \begin{itemize}
 \item $\mu_1$ vanishes on $(E_q\times Q^{-\varepsilon})\cup (S^{-N}\times F_q)\supset \sigma^{-N-\varepsilon}$;
 \item $d\mu_1=\alpha\vee\beta$.
 \end{itemize}
We deduce that $0=j^*(\alpha\vee\beta)\in H^{m+m_0}(\sigma_q^{u(q)-\varepsilon}, \sigma_q^{-N-\varepsilon})$ and then that $U(q)\geq u(q)-\varepsilon$. Hence $U(q)\geq u(q)$ and finally $u(q)=U(q)$.

\end{proof} 

\begin{notas}
  From now we denote by $S: (q, \tau, \xi)\in E\rightarrow S(q, \tau, \xi)$ a special generating function for $\cg$. The critical locus is denoted by $\Sigma$ and the associated embedding is $i=i_S:\Sigma\rightarrow T^*\cm_2$.
  We denote by $u: (q, \tau)\in \cm_2\rightarrow u(q, \tau)$ the graph selector associated to $S$.
\end{notas}

Following the proofs that are contained in \cite{PPS} or \cite{Sib} for the homology, we will prove

\begin{proposition}\label{geneselec}
Let $u:\cm_2\rightarrow \R$ be a graph selector for the special generating function $S:\cm_2\times \R^k\rightarrow \R$.  Then $u$ is a Lipschitz function that is $C^1$ on an open subset $U_0\subset \cm_2$ with full Lebesgue measure, and for every $z=(q,t)\in U_0$, the following properties hold
\begin{equation}
J_u(z)\in \cg,\quad \text{and} \quad u(z)=S\circ i^{-1}(J_u(z)),\end{equation}
where $J_u(z)=(q,d_q u(z),t,\frac{\partial u}{\partial t}(z))\in
T^*M\times\T_2\times\R$ and with the usual identification
$T^* \cm_2=T^*M\times \T_2\times\R$.
\end{proposition}
\begin{proof}
We assume that  $S=Q$ on all the levels that are not in $(-N, N)$ and we denote   the index of $Q$ by $m$. Let us fix $z\in\cm$. We denote by $\alpha$ a $m$-form on $\R^k$ that vanishes on $Q^{-N}$ and is such that $0\not=[\alpha]\in H^m(\R^k, Q^{-N})$. Because there is a change in the topology of the sublevel with height $u(z)$, $u(z)$ is a critical value of $S_z$. \\
Let us prove that $u$ is Lipschitz.  Observe that the function $\cv: \cm\times\cm\times \R^k\rightarrow \R$ that is defined by $\cv(z, z', \xi)=S(z', \xi)-S(z, \xi)$ is $C^1$ and has compact support. Hence there exists a constant $L>0$ such that
$$\forall z, z'\in \cm, \forall\xi\in \R^k, |S(z, \xi)-S(z', \xi)|\leq L.d(z, z').$$
We deduce that for every $a\in \R$, $S_z^{a}\subset S_{z'}^{a+L.d(z, z')}$. Then the inclusion maps induce the following maps (note that $S_z^{-N}=S_{z'}^{-N}$):
$$H^m(\R^k, S_z^{-N})\begin{matrix} j_2^*\\\longrightarrow\\{}\\ \end{matrix}H^m(S_{z'}^{u(z)+L.d(z, z')+\varepsilon}, S_{z'}^{-N})\begin{matrix} j_1^*\\\longrightarrow\\{}\\ \end{matrix}H^m(S_{z}^{u(z)+\varepsilon}, S_{z}^{-N}).$$ We know that $0\not= (j_2\circ j_1)^*\alpha\in H^m(S_{z}^{u(z)+\varepsilon}, S_{z}^{-N})$. This implies that $j_2^*\alpha\not= 0$ and then that $u(z')\leq u(z)+\varepsilon +Ld(z, z')$.
This is also valid when we exchange $z$ and $z'$. Therefore, when we let
$\varepsilon$ go to zero, we get :
$$|u(z)-u(z')|\leq L.d(z, z').$$
Let us now prove that there exists an open subset $U_0$ of $\cm$ with full Lebesgue measure on which $u$ is $C^1$.   Observe that the  set $U_1$ of the $z\in\cm$ where $S_z$ is Morse is exactly the set of regular values of the restriction to $\Sigma_S$ of the first projection $(z, v)\in\cm\times \R^k\mapsto z$ and then has full Lebesgue measure by Sard's theorem.  It is open.  We denote by $U_0$ the set of the $z\in U_1$  such that the critical points of $S_z$   have pairwise distinct critical values.   Let us prove that $U_1\backslash U_0$ has only isolated points: this will imply that $U_0$ is open and has full Lebesgue measure. Let us consider $z\in U_1\backslash U_0$. As $S_z$ is Morse, $\Sigma_S$ is transverse to $\{ z\} \times \R^k$ and then above a neighbourhood $V_z$ of $z$ in $\cm$, $\Sigma_S$ is the union of $j$ graphs, the graphs of $\eta_1, \dots, \eta_j: V_z\rightarrow \R^k$. If we use the notation $\psi_j(z')=\frac{\partial S}{\partial q}(z', \eta_j(z'))$, then $\cg$ is the union of the disjoints graphs of $\psi_1, \dots, \psi_j$ above $V_z$.  For $z'\in V_z$, $u(z')$ is a critical value of $S_{z'}$ and then is one of the real numbers $S(z', \eta_1(z')), \dots, S(z', \eta_j(z'))$. Note that every $S(, \eta_i(.))$ is $C^1$ and that $\frac{\partial S(, \eta_i(.))}{\partial z}=\psi_i$. As the $\psi_i(z')$ are pairwise distinct, for $i\not=j$,   $\{ S(, \eta_i(.))=S(, \eta_j(.))\}$ has only isolated points.\\
Let us now consider $z\in U_0$. We can define a connected neighbourhood $V_z$, $\eta_1, \dots , \eta_j$ and $\psi_1, \dots , \psi_j$ exactly as before.  Then every $u(z')$ is one of the  $S(z', \eta_i(z'))$. Because $V_z$ is a connected part of $U_0$, there exists exactly one $i$ such that $\forall z'\in V_z, u(z')= S(z', \eta_i(z'))$. Then we have
$du(z')= \frac{\partial S}{\partial z}(z', \eta_i(z'))=\psi_i(z')$ and we deduce that
$u(z')=S\circ i_S^{-1}(J_u(z'))$ and $J_u(z')\in \cg$.
\end{proof}

\begin{proposition}\label{ident-u}
There exist a real constant $c$ such that the following identity holds
\begin{equation}
\forall q \in M,\quad u(q,1)=u(q,0)-c 
\end{equation}
\end{proposition} 
\begin{proof}
We proved in Proposition \ref{P2generating} that $S(q, 0; \xi)$ and $S(q, 1,; \xi)$ are two generating functions for $\cl$.  We deduce from Proposition \ref{unicselector} the wanted result.
\end{proof}
\begin{cor} \label{identit-S}
For the same constant $c$ that is defined in Proposition \ref{ident-u}, the function $S\circ i_S^{-1}$ satisfies the identity 
\begin{equation}\label{eqgene}
S\circ i_S^{-1}(q,p,1,0)=S\circ i_S^{-1}(q,p,0,0)-c, \qquad (q,p)\in \cl.
\end{equation}
\end{cor}
 
\begin{proof}
As $S(q, 0; \xi)$ and $S(q, 1; \xi)$ are two generating functions for $\cl$, the functions $(q, p)\in \cl\mapsto S\circ i^{-1} (q, p, 1, 0)$ and $(q, p)\in \cl\mapsto S\circ i^{-1} (q, p, 0, 0)$  are two primitive on $\cl$ of the Liouville $1$-form $\lambda$.  Hence their difference is a constant. \\
Moreover, $u(., 0)$ and $u(., 1)$ are two graph selectors for $\cl$ so that $u(., 0)-u(., 1)=c$. Hence there exists a dense open subset $V_0$ of $M$ with full Lebesgue measure such that for $i=1, 2$
$$\forall q\in V_0, (q, d_qu(q, 0))=(q, d_qu(q, 1))\in \cl\quad{\rm and}\quad u(q, i)=S\circ i_S^{-1}(q, d_qu(q, i), i, 0).$$
Take $q\in V_0$. By Proposition \ref{geneselec}, we have for $(q, p)=(q, d_qu(q, 0))=(q, d_qu(q, 1))\in\cl$ that
$$S\circ i_S^{-1}(q, p, 1, 0)=u(q, 1)=u(q, 0)-c=S\circ i_S^{-1}(q, p, 0, 0)-c.$$

\end{proof}
%
%
%
%
%
\section{Construction of a dominated function.}\label{domin}
In this section, we come back  to the original problem, and construct what is called a {\sl dominated function} for the Lagrangian $L$ that is associated to $H$, where we recall the definition of the Lagrangian that we gave in the introduction.

\begin{nota}
The {\em Lagrangian} $L:TM\times \T\rightarrow \R$ is the function that is associated to $H$ via the Legendre duality.

$$\forall (q, v, t)\in TM\times \T, L(q, v, t)=\inf_{p\in T^*_qM}(p.v-H(q, p, t)).$$
\end{nota}
We recall that $L$ is as regular as $H$ is, $C^2$-convex in the fiber direction and superlinear in the fiber direction (see e.g. \cite{Fa1}).

\begin{defi}
A function $U:\cm_1=M\times \T\rightarrow \R$ is {\em dominated} by $L+c$ if it is Lipschitz and if for every continuous and piecewise $C^1$ arc $\gamma:[a, b]\rightarrow M$, we have
$$U(\gamma (b), b)-U(\gamma (b), b)\leq \int_a^b(L(\gamma (t), \dot\gamma (t), t)+c)dt.$$
\end{defi}

The goal of this section is to build a function $\mathfrak{u}$ that is dominated by $L+c$ and to prove some properties for this function.

Then, in the last section, we will prove that 
$\ug$ is everywhere differentiable and that $\cl$ is contained in the graph of $q\mapsto d\ug(q, 0)$. After that, we will prove that $d\ug$ is $C^1$.

\subsection{Construction of a dominated function.} Let us introduce a notation.
\begin{nota}
We define $\ug : M\times [0, 1]\rightarrow \R$ by $\ug(q, t)= u(q,\eta^{-1}(t))+ct$.
\end{nota}
Observe that a consequence of Proposition \ref{ident-u} is that $\ug(., 0)=\ug(., 1)$. Hence we can consider $\ug$ as a function defined on $\cm_1=M\times \T$. 
\medskip


\begin{proposition}
The function $\ug$ is Lipschitz and dominated by $L+c$.
\end{proposition}
\begin{proof}
We postpone the proof that $\ug$ is Lipschitz after the proof of the domination property, but we use the fact that $\ug$ is Lipschitz in the first part of our proof.\\

{\sl The domination property}\\
Let $\gamma:[a, b]\rightarrow M$ be a $C^1$ arc with $[a, b]\subset (0, 1)$ and assume that  the image of $t\in [a, b]\mapsto (\gamma\circ \eta (t), t)\in \cm_1$ is   Lebesgue almost everywhere in $U_0$ ($U_0$ was defined in Proposition \ref{geneselec}). Then

$$\ug(\gamma (b), b)-\ug(\gamma(a),a)=u(\gamma(b), \eta^{-1}(b))-u(\gamma(a), \eta^{-1}(a))+c(b-a),$$
and  if we use the notation $\delta=u(\gamma(b), \eta^{-1}(b))-u(\gamma(a), \eta^{-1}(a))$
$$\delta=\int_{\eta^{-1}(a)}^{\eta^{-1}(b)}\left(d_qu(\gamma(\eta(t)), t)\dot\gamma(\eta(t))\dot\eta(t)+\frac{\partial u}{\partial t}(\gamma\circ \eta(t), t)\right)dt.$$
Young inequality for dual convex functions tells us that
$$\forall p\in T^*_qM, \forall v\in T_qM, \forall t\in \T, p.v\leq H(q, p,t)+L(q, v,t).$$
Hence  we have $\delta\leq $\\
$$\int_{\eta^{-1}(a)}^{\eta^{-1}(b)}\left[\dot\eta(t)\big(H(\gamma(\eta(t)),d_qu(\gamma(\eta(t)),t),\eta(t))+L(\gamma(\eta(t)),\dot\gamma(\eta(t)),\eta(t))\big)+\frac{\partial u}{\partial t}(\gamma\circ \eta(t), t)\right]dt.$$
Proposition  \ref{geneselec} tells us that $u$ is a graph selector for $\cg$ above $U_0$. We can therefore replace in the integral
$\frac{\partial u}{\partial t}(\gamma\circ \eta(t), t)$ by
$-K(\gamma(\eta(t)), p(t), t)=-\dot\eta(t)H(\gamma(\eta(t)), p(t), \eta(t))$,
where we set $p(t)=d_qu(\gamma(\eta(t),t)$. Using a change of variable $s=\eta(t)$
we obtain 
$$\delta \leq \int_{\eta^{-1}(a)}^{\eta^{-1}(b)}\dot\eta(t)L(\gamma(\eta(t)),\dot\gamma(\eta(t)),\eta(t))dt=\int_a^bL(\gamma(s), \dot\gamma(s), s)ds.$$
This gives the domination property
\begin{equation}
\label{eqdom}
\ug(\gamma (b), b)-\ug(\gamma(a),a)\leq \int_a^b(L(\gamma(s), \dot\gamma(s), s)+c)ds.\end{equation}
How can we conclude for general $\gamma: [a, b]\rightarrow M$ that are continuous and piecewise $C^1$?
\begin{itemize}
\item if  $[a,b]\subset (0, 1)$ and $\gamma$ is $C^1$, by Lemma
  \ref{arg-Fubini} below applied to
  $\gamma\circ \eta$,  we can approximate $\gamma$  in topology $C^1$
  by a sequence $(\gamma_n)_n$ such that $(\gamma_n\circ\eta(t),t)\in U_0$ for almost every
  $t\in [\eta^{-1}(a),\eta^{-1}(b)]$, hence the domination inequality holds for every $\gamma_n$.
  Taking now the limit $n\rightarrow +\infty$, we find that inequality (\ref{eqdom})
holds for our curve $\gamma$.
\item if $[a, b]\subset [0, 1]$ and $\gamma$ is $C^1$, we can find a decreasing sequence $(a_n)$ and an increasing sequence $(b_n)$ so that $(a, b)=\bigcup_{n\in\N}[a_n, b_n]$; then every $\gamma_{|[a_n, b_n]}$ is dominated and by taking a limit $\gamma$ is dominated;
\item for general $\gamma$, we can cut $\gamma$ in sub-arcs $\gamma_1, \dots, \gamma_n$ that are $C^1$ and defined on some intervals $I_k$ that are contained in some intervals $[n_k, n_k+1]$ with $n_k\in \Z$; then we have the domination property for every $\gamma_j$ and hence for their concatenation $\gamma=\gamma_1*\dots *\gamma_n$.
\end{itemize}
\vspace{1mm}  \noindent
\begin{lemma} \label{arg-Fubini}
  Given an interval $[\alpha,\beta]\subset (0,1)$, a set of full measure $U_0\subset \cm_1$ and a
  ${C}^1$ curve $\tau : [\alpha,\beta]\rightarrow M$, there exists a sequence of ${C}^1$
  curves $\tau_n :[\alpha,\beta]\rightarrow M$, $n\in \N$ such that $(\tau_n)_{n\in\N}$ converges to $\tau$
  in the ${C}^1$-topology, and for every $n\in \N$, $(\tau_n(t),t)\in U_0$ for almost every
  $t\in [\alpha,\beta]$.
\end{lemma}  
\begin{proof}
  Without loss of generality, we can assume that $\tau$ is defined in a slightly bigger interval
  $[\alpha^\prime,\beta^\prime]\subset (0,1)$ such that $[\alpha,\beta]\subset (\alpha^\prime,\beta^\prime)$.
  The curve
  $\sigma : [\alpha^\prime,\beta^\prime]\rightarrow \cm_1$, $\sigma(t)=(\tau(t),t)$ is a
  ${C}^1$-embedding, and it can be embedded in a tubular neighbourhood, that is to say,
  there exist a
  ${C}^1$-embedding 
  $\Lambda : [\alpha^\prime,\beta^\prime]\times {\mathcal O}\rightarrow \cm_1$,
  $\Lambda(t,\xi)=(\Gamma(t,\xi),T(t,\xi))$ such that $\Gamma(t,0)=\tau(t)$ and $T(t,0)=t$,
  where ${\mathcal O}$ is on open neighbourhood of $0$ in $\R^n$. Let us prove that it is always possible to find a
  tubular  neighbourhood $\tilde{\Lambda}$ of $\sigma$ so that
  $\tilde{\Lambda}(t,\xi)=(\tilde{\Gamma}(t,\xi),t)$. 
  Indeed,  
  let $F :[\alpha^\prime,\beta^\prime]\times {\mathcal O} \rightarrow \R\times {\mathcal O}$ be the map
  defined by $F(t,\xi)=(T(t,\xi),\xi)$. Since  $T(t,0)=t$, the differential 
   $D_{(t,0)} F$ is the identity. Eventually shrinking  ${\mathcal O}$ we can assume that $D_{(t,\xi)} F$
  is invertible for every  $(t,\xi)\in [\alpha^\prime,\beta^\prime]\times {\mathcal O}$. Since 
  the map $t\mapsto T(t,\xi)$ is ${C}^1$-close  to the identity for $\xi$ sufficiently small,
  hence injective, we can also assume that $F$ is injective, and therefore $F$ defines a
  diffeomorphism from $[\alpha^\prime,\beta^\prime]\times {\mathcal O}$ to a neighbourhood of
  $[\alpha,\beta]\times \{0\}$. By definition of $F$, if we set $\tilde{\Lambda}=\Lambda\circ F^{-1}$
  we get
  $\tilde{\Lambda}(t,\xi)=(\tilde{\Gamma}(t,\xi),t)$, where $\tilde{\Gamma}=\Gamma\circ F^{-1}$.
  Now $F([\alpha^\prime,\beta^\prime]\times {\mathcal O})$ is a neighbourhood of
  $[\alpha,\beta]\times\{0\}$, hence we can find an open neighbourhood of $0$ in $\R^n$, here denoted
  $\tilde{\mathcal O}$, such that
  $[\alpha,\beta]\times \tilde{\mathcal O}\subset F([\alpha^\prime,\beta^\prime]\times {\mathcal O})$.
  Since $\tilde{\Lambda}$ is a ${C}^1$-diffeomorphism, the set
  $V_0=\tilde{\Lambda}^{-1}(U_0\cap \tilde{\Lambda}([\alpha,\beta]\times\tilde{\mathcal O}))$ has full
  measure in
  $[\alpha,\beta]\times \tilde{\mathcal O}$, and by Fubini Theorem, for almost every
  $\xi\in \tilde{\mathcal O}$, the set of  $t\in [\alpha,\beta]$ such that $(t,\xi)\in V_0$ has
  full measure in $[\alpha,\beta]$, therefore, we can
  find a
  sequence $(\xi_n)_n$ in $\tilde{\mathcal O}$ such that $\xi_n\rightarrow 0$ and for almost every
  $t\in [\alpha,\beta]$
  we have $(\tilde{\Gamma}(t,\xi_n),t)\in U_0$. By defining $\tau_n(t)=\tilde{\Gamma}(t,\xi_n)$ we have
  the desired property.
\end{proof}
\medskip

{\sl The Lipschitz property}\\

Let us remark that $\eta^{-1} : [0,1]\rightarrow [0,1]$ is an absolutely continuous function.  Indeed, it is a 
${C}^\infty$ function on the open interval $(0,1)$, and if we set $g(t)=(\eta^{-1})^\prime(t)$ for $t\in (0,1)$, for every 
segment $[a,b]\subset (0,1)$ we have 
\begin{equation} \label{abs-cont}
\eta^{-1}(b)-\eta^{-1}(a)=\int_a^b g(t)\, dt,
\end{equation}  \\
and by construction of $\eta$ we know that $g(t)>0$ for $t\in (0,1)$. By continuity of $\eta^{-1}$, if we take the limits 
$a\rightarrow 0$ and $b\rightarrow 1$, we find that $g$ is absolutely integrable on $(0,1)$, and identity (\ref{abs-cont}) 
holds for every $[a,b]\subset [0,1]$, hence $\eta^{-1}$ is absolutely continuous.   
As $u$ is Lipschitz, the function $\ug$ that we defined by $\ug(q, t)= u(q,\eta^{-1}(t))+ct$ is (uniformly) absolutely continuous 
in the $t$-direction and (uniformly) Lipschitz in the $q$ direction. Hence, to prove that $\ug$ is Lipschitz, we just have to prove that its derivative, which is defined Lebesgue almost everywhere, is bounded on a set with full Lebesgue measure.

Observe that for every segment $[a,b]\subset (0, 1)$, the map
$\eta\left|_{[a,b]}\right. : [a,b] \rightarrow [\eta(a),\eta(b)]$ is a bi-lipschitz homeomorphism;
we deduce that the set
\begin{equation} \label{dfu0}
\cu_0=\{ (q, \eta(t))); (q, t)\in U_0\cap (M\times (0, 1))\}
\end{equation}
has full Lebesgue measure in $\cm_1=M\times \T$.
For $(q, t)\in \cu_0$, we have $(q, \eta^{-1}(t))\in U_0$ and then $(q, \eta^{-1}(t), du(q, \eta^{-1}(t)))\in\cg$. This implies that $d_qu(q, \eta^{-1}(t))$ is (uniformly) bounded on $\cu_0$ and 
$$\frac{\partial u}{\partial t}(q, \eta^{-1}(t))=-K(q, d_qu(q, \eta^{-1}(t)), t)=-\dot\eta(\eta^{-1}(t))H(q, d_qu(q, \eta^{-1}(t)), t).$$
We deduce that
 
\begin{equation}d\ug(q,t)(\d q, \d t)= d_qu(q, \eta^{-1}(t))\d q+\frac{1}{\dot\eta(\eta^{-1}(t))}\frac{\partial u}{\partial t}(q, \eta^{-1}(t))\d t\end{equation}
is equal to 
$$d\ug(q,t)(\d q, \d t)= d_qu(q, \eta^{-1}(t))\d q-H(q, d_qu(q, \eta^{-1}(t)), t)\d t
$$
and thus $d\ug$ is bounded above $\cu_0$.\\
Let us now conclude.
Given now two points $(q,t)$ and $(q^\prime,t^\prime)\in \cm_1$, we have
\begin{equation} \label{lip-ug}
\begin{array}{rl}
  |\ug(q^\prime,t^\prime)-\ug(q,t)|\le |&\ug(q^\prime,t^\prime)-\ug(q,t^\prime)|+|\ug(q,t^\prime)-\ug(q,t)| \\
  &\le A\, dist(q^\prime,q)+|\ug(q,t^\prime)-\ug(q,t)|
\end{array}
\end{equation}
where $dist(\, ,\,)$ is a Riemannian distance on $M$, $A$ is a positive constant independent from $(q,t)$
and $(q^\prime,t^\prime)$. By an argument similar to
the one given in proof of Lemma \ref{arg-Fubini}, and eventually cutting the segment $s\mapsto (q,s)$
in a finite number of pieces, we can find a sequence of points $(q_n)_n$ in $M$ converging to $q$ and
such that for every $n\in \N$, the point $(q_n,s)$ is in $\cu_0$ for almost every
$s\in [t,t^\prime]$ (without loss of generality we assume $t<t^\prime$). Since we know that $d\ug$ is
bounded above $\cu_0$ we find
$$
|\ug(q_n,t^\prime)-\ug(q_n,t)| \le \displaystyle
\int_t^{t^\prime} \left| \frac{\partial \ug}{\partial t}(q_n,s)\right|\, ds \le B\, |t^\prime-t|,
$$
for some constant $B>0$. Taking now the limit $n\rightarrow +\infty$ and replacing in (\ref{lip-ug})
we finish the proof. 
\end{proof}
\subsection{The dominated function $\ug$ can be seen as a kind of graph selector}

In this part, we construct an extended Hamiltonian of $H$ and an extended Lagrangian submanifold $\cy$ of $\cl$ by using $H$. We will prove that in some sense, $\ug$ is a graph selector for $\cy$.

\begin{nota} We introduce the autonomous Hamiltonian
$\ch$ on 
$T^*\cm_1=T^*M\times \T_1\times \R$ that is defined by 
$$
\ch(q,p,\tau,e)=H(q,p,\tau)+e.
$$ 
\end{nota}
The Hamiltonian equations for $\ch$ are
\begin{equation} \label{champ_hamiltonien_ch}
\left\{
\begin{array}{rlrl}
\frac{dq}{dt}&=\frac{\partial H}{\partial p}(q,p,\tau), \qquad 
&\frac{d\tau}{dt}&=1, \\
\frac{dp}{dt}&=-\frac{\partial H}{\partial q}(q,p,\tau), \qquad
&\frac{de}{dt}&=-\frac{\partial H}{\partial \tau}(q,p,\tau),
\end{array}
\right.
\end{equation}
and the flow of (\ref{champ_hamiltonien_ch}) is given by 
\begin{equation} \label{flow_ch}
\phi_\ch^t(q,p,\tau,e)=(\varphi_H^{\tau,\tau+t}(q,p),\tau+t,
e+H(q,p,\tau)-H(\varphi_H^{\tau,\tau+t}(q,p),\tau+t)). 
\end{equation}
If we denote by $F_E(q, p, \tau, e)=(q, p, \tau, e+E)$ the translation in the energy direction by $E$, observe that $F_E\circ \phi_\ch^t=\phi_\ch^t\circ F_E$. Hence the restriction of $(\phi_\ch^t)$ to every level $\{Ê\ch=E\}$ is conjugated (via $F_E$) to the restriction of $(\phi_\ch^t)$ to the zero level $\{ \ch=0\}$.


Similarly to what we did in 
the previous section for the construction of $\cg$, we now extend  $\cl$ to a 
Lagrangian submanifold $\cy$ of $T^*\cm_1$ invariant by the 
flow $(\phi_{\ch}^t)$. The only change is that we choose the lift in such a way that $\cy\subset \{ \ch=c\}$ for the constant $c$ that we introduced in Proposition \ref{ident-u} and Corollary \ref{identit-S}.

$$
\cy=  \{ \phi_\ch^t (q,p,0,-H(q,p,0)+c); (q,p)\in \cl,  \ t\in [0,1]\}.
$$  
Since $\cl$ is invariant by $\varphi_H^{0,1}$,   $\cy$ 
is a closed submanifold of 
$T^*\cm_1$. Observe that  $\cy$ is contained in the energy level 
$\{ \ch=c\}$. 
\begin{proposition}\label{PLiouvilleprim}
The manifold $\cy$ is exact Lagrangian, i.e. the Liouville 1-form $\theta=<p,dq>+Ed\tau$ has a primitive $\cs$ along $\cy$.
\end{proposition}
\begin{proof}
Let $\tilde\cy$ be the set 
$\tilde\cy=\{ (q, p, t ,e)\in T^*M\times [0, 1]\times \R; (q, p, [t]_1, e)\in \cy\}$.

We define the map $\psi: T^*M\times [0,1]\times\R \rightarrow T^*\cm_2$ by
\begin{equation}\psi (q, p, t, e)=(q, p, \eta^{-1}(t), \dot\eta(\eta^{-1}(t))(e-c)).\end{equation}

\begin{lemma} \label{homeo-psi}
  $\psi\left|_{\tilde\cy}\right.$ is an homeomorphism from $\tilde \cy$ onto
  $\cg\cap (T^*M\times [0, 1]\times \R)$.
\end{lemma}
\begin{proof}
  Let $(q, p, t, e)\in \tilde \cy$. This means that $e=c-H(q, p, t)$ and $\varphi_H^{t, 0}(q, p)\in \cl$.
  Then $\psi (q, p, t, c-H(q, p, t))=(q, p, \eta^{-1}(t), -K(q, p, \eta^{-1}(t)))$ with $\varphi_K^{\eta^{-1}(t), 0}(q, p)\in\cl$. Hence $\psi (\tilde\cy)$ is $\cg\cap (TM\times [0, 1]\times \R)$.\\
The continuity and injectivity are straightforward.
\end{proof}

We define then $s_0$ by
$$s_0(q, p, t, e)=S\circ i_S^{-1}\circ\psi(q, p,t, e)+ct.$$
Because of equality (\ref{eqgene}), we have
$$s_0(q, p, 1, e)=S\circ i_S^{-1}(q, p, 1,0)+c=S\circ i_S^{-1}(q, p, 0, 0)=s_0(q, p, 0, e).$$
 Hence we can define $\cs: \cy\rightarrow \R$ by $\cs(q, p, [t]_1, e)=s_0(q, p, t, e)$.

This function $\cs$ is continuous on $\cy$ and is differentiable except on the slice $\cy\cap \{t=0\}$. We have $d\cs (q, p, t,e)(\d q, \d p, \d t, \d e)= $
$$d(S\circ i_S^{-1})(q, p, \eta^{-1}(t), \dot\eta(\eta^{-1}(t))(e-c))(\delta q, \delta p, \frac{1}{\dot\eta(\eta^{-1}(t))}\delta t, \delta E)+c\delta t,$$
with $\delta E=\dot\eta (\eta^{-1}(t))\delta e+(e-c)\frac{\ddot \eta(\eta^{-1}(t))}{\dot \eta(\eta^{-1}(t))}\delta t$.\\
As $S\circ i_S^{-1}$ is a primitive of the Liouville 1-form $\Theta=<p, dq>+Ed\tau$, we deduce that 
$$d\cs (q, p, t,e)(\d q, \d p, \d t, \d e)= <p, \delta q>+\dot\eta (\eta^{-1}(t))(e-c)\frac{\delta t}{\dot\eta(\eta^{-1}(t))}+c\delta t=<p, \delta q>+e\delta t.$$
Hence $\cs$ is continuous on $\cy$ and is a primitive of $\theta$ on $\cy\backslash \{ t=0\}$. 

As $\cy$ is Lagrangian, a primitive of $\theta$ along $\cy$ exists always locally and is $C^1$. Then for every point in $\cy$, there exists a connected open neighborhood $\cv$ on which $\theta$ has a $C^1$ primitive $s$. Without loss of generality we can assume that $\cv\backslash\{ t=0\}$ is made by one or two  (open) connected 
components $\cv_1$, $\cv_2$ (that may be equal). Observe that $\overline{\cv_1}\cap \overline{\cv_2}\not=\emptyset$. On each of these open and connected components $\cv_i$, $\cs-s$ is differentiable with its differential equal to $0$, hence $(\cs-s)_{|\cv_i}$ is  equal to a constant $c_i$. As $\cs-s$ is continuous, we have also $(\cs-s)_{|\overline{\cv_i}}=c_i$. As $\overline{\cv_1}\cap \overline{\cv_2}\not=\emptyset$, we have $c_1=c_2$ ad then  $\cs-s$ is constant on $\cv$, therefore $\cs$ is $C^1$ everywhere and is a primitive of the Liouville 1-form $\theta$.
 
\end{proof}
As the exact Lagrangian $\cg$ has a graph selector, the same is true for $\cy$.
\begin{proposition} \label{u-lipschitz-select}
  The function $\ug$ is differentiable at every $z=(q,t)\in \cu_0$, where
  $\cu_0$ is the open subset defined in (\ref{dfu0})
  and moreover
\begin{equation} \label{kooks}
\forall z\in \cu_0,\quad J_{\ug}(z)\in \cy \quad \text{and}\quad \ugoth(z)=\cs(J_{\ug}(z)),
\end{equation}
where $J_{\ug}(z)=(q,d_q \ug(z),t,\frac{\partial \ug}{\partial t}(z))\in
T^*M\times\T \times\R=T^*\cm_1$.
Moreover identity (\ref{kooks}) holds for every $z\in \cm_1$ where 
$\ugoth$ is differentiable and where $\ch(J_{\ug}(z))=c$.

\end{proposition} 
\begin{proof} 
{\em  Proof that $\ug$ is a graph selector.}\\
\begin{nota}
Let $h: \cm_1\rightarrow \cm_1$ be defined by $h(q, t)=(q, \eta^{-1}(t))$.
\end{nota}

Observe that $h$ is an homeomorphism and that $h_{|M\times (0, 1)}$ is a diffeomorphism onto $M\times (0, 1)$.
By definition, the function $\ugoth$ is differentiable on every $z=(q,t)\in \cu_0$ and
moreover
$$
J_{\ug}(z)=(q,d_q u(h(z)),t,c+\frac{1}{\dot{\eta}(\eta^{-1}(t))} \frac{\partial u}{\partial t}(h(z))).
$$
As $u$ is a graph selector for $\cg$, we have
$$
\psi(J_{\ug}(z))=(q,d_q u(h(z)),\eta^{-1}(t),\frac{\partial u}{\partial t}(h(z)))
=J_u(h(z))\in \cg\cap (T^*M \times (0,1)\times \R).
$$
By construction of $\psi$ and by Lemma \ref{homeo-psi}, we can say that
$\psi$ maps $\cy\cap (T^*M\times (0,1)\times \R)$ diffeomorphically onto
$\cg \cap (T^*M \times (0,1)\times \R)$, therefore 
$J_{\ug}(z)\in \cy$.\\
Moreover we have
$$\ug (z)=u(h(z))+ct=S\circ i_S^{-1}(J_u(h(z)))+ct=S\circ i_S^{-1}\circ\psi(J_{\ug}(z))+ct=
\cs(J_{\ug}(z)).$$

\medskip

{\sl Proof that Identity (\ref{kooks}) holds for every $z\in \cm_1$ where 
$\ugoth$ is differentiable and where $\ch(J_{\ugoth}(z))=c$.}\\

Let $z=(q,t)\in \cm_1$ be a point where $\ugoth$ is differentiable and 
$\ch(J_\ug(z))=c$. We follow the same step as in \cite{Arn}, and we
introduce two subsets of 
$T^*_{z}\cm_1=T_q^*M\times\R$.
Let $K_{\ugoth}(z)$ 
be the set of all limit points of sequences $(d\ugoth(z_n))_{n\in\N}$ where
$z_n\in \cu_0$ and 
$\lim\limits_{n\rightarrow +\infty} z_n=z$, and let $C_{\ugoth}(z)$ be the convex
hull of 
$K_{\ugoth}(z)$.  Let us give  a result due to F.~Clarke (see \cite{FM} for a proof of \cite{Cla} for a more general result).
\begin{lemma}\label{nonsmooth}
Let $f~: U\rightarrow \R$ be a Lipschitz function defined on a open subset $U$ of $\R^d$ and let $U_0\subset U$ be a subset with full Lebesgue measure such that $f$ is differentiable at every point of $U_0$. We introduce a notation. If $q\in U$, $K_f(q)$ is the  set of all the limits $\displaystyle{\lim_{n\rightarrow \infty} df(q_n)}$ where $q_n\in U_0$,  $\displaystyle{\lim_{n\rightarrow \infty}q_n=q}$ and $C_f(q)$ is the convex hull of $K_f(q)$. Then, at every point $q\in U$ where $f$ is differentiable, we have~: 
$df(q)\in C_f(q)$.
\end{lemma}

 By hypothesis the function $p\mapsto H(q,p,t)$ 
is strictly convex, therefore the {\sl energy sublevel }  
$$
\ch^{-1}_{(q,t)}((-\infty,c])=\left\{ (p,e)\in T^*_{(q,t)} (M\times\T_1),\quad 
H(q,p,t)+e\le c \right\}
$$
is also strictly convex (up to the symmetry $e\mapsto -e$, it is the epigraph 
of the function $p\mapsto H(q,p,t)-c$), and in particular, every point 
$(p,e)$ in the 
{\sl energy level} $\ch^{-1}_{(q,t)}(c)$ is extremal for 
$\ch^{-1}_{(q,t)}((-\infty,c])$. By hypothesis, $d\ugoth(z)$ is in the 
energy level $\ch^{-1}_z(c)$, therefore it is an extremal point of  
$C_\ugoth(z)$, hence a point of $K_{\ugoth}(z)$, and by definition of 
$K_{\ugoth}(z)$ there exist a sequence $(z_n)_{n\in\N}$ of points of $\cu_0$  
such that $(z_n,d\ugoth(z_n))_{n\in\N}$ converges to $(z,d\ugoth(z))$, or equivalently
$(J_\ug(z_n))_{n\in\N}$ converges to $(J_\ug(z))$,
but every 
point $J_\ug(z_n)$ lies in $\cy$ and satisfies the identity 
$\ugoth(z_n)=\cs(J_\ug(z_n))$. Taking the limit $n\rightarrow +\infty$ 
we get (\ref{kooks}). 
\end{proof}
\section{Calibration and conclusion}
In this section, we will prove that $\cy$ is contained in the graph of $d\ug$. \\
Observe that 
\begin{itemize}
\item the projection of $\cy$ is compact because $\cy$ is compact;
\item the projection of $\cy$ is dense in $\cm_1$. Indeed, Proposition \ref{u-lipschitz-select} implies that this projection contains $\cu_0$, which is dense in $\cm_1$.
\end{itemize}
Hence the projection of $\cy$ is $\cm_1$ and we will conclude that 
$\cy$ is a graph above the whole $\cm_1$ and that $\ug$ is everywhere differentiable. Thus $\cy$ is the  the graph of $d\ug$.\\
 Morever, we will also prove that $\cy$ is a locally Lipschitz graph in $T^*\cm_1$. Hence $\cy$ is a $C^1$ manifold that is the graph of a locally Lipschitz map. As $\cm$ is compact, this implies that $\cy$ is the graph of a $C^1$ map, i.e. that $\ug$ is $C^2$ and $\cy$ is the graph a $C^1$ exact 1-form: $d\ug$.

The main tool that we will use is the notion of calibrated curve.
\subsection{Calibration}
We will explain what happens along the curves that satisfy the equality in the inequality of domination (\ref{eqdom}). The proof is an analogue of the proof given by A.~Fathi in \cite{Fa1} in the autonomous case.
\begin{defi}
If $\gamma:[a, b]\rightarrow M$ is a $C^1$ arc, its {\em defect of calibration} is
$$\delta (\gamma)=\int_a^b(L(\gamma(t), \dot\gamma(t), t)+c)-(\ug(\gamma(b), b)-\ug(\gamma(a), a)).$$
\end{defi}
Then
\begin{itemize}
\item $\delta$ is always non-negative;
 \item if $(\gamma_n)$ $C^1$-converges to $\gamma$, then $\displaystyle{\lim_{n\rightarrow \infty} \delta(\gamma_n)=\delta(\gamma)}$;
 \item if $I\subset J$, then $\delta(\gamma_{|I})\leq \delta (\gamma_{|J})$.
 \end{itemize}
 
 \begin{defi}
 A $C^1$ curve $\gamma:I\rightarrow M$ is {\sl $(\ug, L, c)$-calibrated} if $\forall [a, b]\subset I$, $\delta(\gamma_{|[a, b]})=0$.
 \end{defi}
 
 \begin{proposition}\label{propcal}
  If $\gamma: I\rightarrow M$ is $(\ug, L, c)$-calibrated, then 
\begin{itemize}
\item $\ug$ is differentiable at every $(\gamma(t),t)$ with $t$ in the interior of $I$;
\item for all $t$ in the interior of $I$, we have  $d_q\ug(\gamma(t), t)=\frac{\partial L}{\partial v}(\gamma(t), \dot\gamma(t), t)$ and $\ch (J_\ug(\gamma(t),t))=c$.
\end{itemize}
 \end{proposition}
 
 \begin{proof} We assume that $\gamma$ is $C^1$ and calibrated.\\
 
{\sl Value of $d\ug(\gamma(.), .)$ if $\ug$ is differentiable along $t\mapsto (\gamma(t), t)$.}\\
Let us assume that $u$ is differentiable at every point of $\{ (\gamma(t), t); t\in (a, b)\}$.\\
We have
$$\forall t\in (a, b), \ug(\gamma(t),t)-\ug(\gamma (a),a))=\int_{a}^t(L(\gamma(s), \dot\gamma(s),s)+c)ds.$$
Differentiating with respect to $t\in (a,b) $, we obtain 
$$d_q\ug(\gamma(t), t).\dot\gamma(t)+\frac{\partial \ug}{\partial t}(\gamma(t), t)=L(\gamma(t), \dot\gamma(t), t)+c.$$
Using Young inequality, we deduce
\begin{equation}\label{equalcal}
\begin{matrix}
c&=d_q\ug(\gamma(t), t).\dot\gamma(t)+\frac{\partial \ug}{\partial t}(\gamma(t), t)-L(\gamma(t), \dot\gamma(t), t)\hfill\\
&\leq \frac{\partial \ug}{\partial t}(\gamma(t), t)+H(\gamma(t), d_q\ug(\gamma(t), t),t)=\ch(\gamma(t), t, d\ug(\gamma(t), t)).
\end{matrix}
\end{equation}
But Lemma \ref{nonsmooth} implies that $\ch(\gamma(t), t, d\ug(\gamma(t), t))\leq c$. Hence Inequality (\ref{equalcal}) is in fact an equality. Il particular we have   equality in Young inequality
$$d_q\ug(\gamma(t), t).\dot\gamma(t)=L(\gamma(t), \dot\gamma(t), t)+H(\gamma(t), d_q\ug(\gamma(t), t),t)$$
then $d_qu(\gamma(t), t)=\frac{\partial L}{\partial v}(\gamma(t), \dot\gamma(t), t)$ 
and so $\frac{\partial \ug}{\partial t}(\gamma(t), t)=c-H(\gamma(t), d_q\ug(\gamma(t), t),t)$.
This can be written $\ch (J_{\ug}(\gamma(t), t))=c$.
 
 \medskip
 
 {\sl Proof that $\ug$ is differentiable at every $(\gamma(t),t)$ with $t$ in the interior of $I$}\\
 Let us fix $t_0\in (a, b)$. We work in a chart around $\gamma(t_0)=x$. Then for every $t\in (a, b)$ and $y$ close
 to $x$, we consider the arc $\gamma_{y,t}: [a, t]\rightarrow M$ that is defined by $\gamma_{y,t}(s)=\gamma(s)+\frac{s-a}{t-a}(y-\gamma(t))$. The domination property implies that 
 $$\ug(y, t)\leq \psi_+(y, t)=\ug (\gamma (a), a)+\int_a^t(L(\gamma_{y,t}(s), \dot \gamma_{y,t}(s), s)+c)ds.$$
 Observe that $\ug(x, t_0)=\ug(\gamma (a), a)+\int_a ^{t_0}(L(\gamma(s), \dot\gamma (s), s)+c)ds=\psi_+(x,t_0)$ because $\gamma$ is calibrated. Observe to that
 $$\psi_+(y, t)=\ug (\gamma (a), a)+\int_a^t(L(\gamma(s)+\frac{s-a}{t-a}(y-\gamma(t)), \dot \gamma(s)+\frac{1}{t-a}(y-\gamma(t)), s)+c)ds$$
 and thus $\psi_+$ is $C^1$.
 
 Let us now consider the arc $\eta_{y, t}:[t, b]\rightarrow M$ that is defined by  $\eta_{y,t}(s)=\gamma(s)+\frac{b-s}{b-t}(y-\gamma(t))$. Then 
 $$\psi_-(y, t)=\ug (\gamma(b), b)-\int_t^b(L(\eta_{y, t}(s), \dot \eta_{y, t}(s), s)+c)ds\leq \ug(y, t).$$
 $\psi_-$ is $C^1$ and because $\gamma$ is calibrated we have $\psi_-(x, t_0)=\ug(x, t_0)$. 
 
 Finally, we have found two $C^1$ function $\psi_-$ and $\psi_+$ such that $\psi_-\leq \ug\leq \psi_+$ and $\psi_-(x, t_0)=\ug(x, t_0)=\psi_+(x, t_0)$. This implies that $\ug$ is differentiable at $(x, t_0)$.
\end{proof}
\subsection{Study along the $\Omega$-limit set of $\varphi_{H|\cl}^{0,1}$}

\begin{proposition}\label{POmega}
  Let $(q, p)\in \Omega (\varphi_{H|\cl}^{0,1})$ be a point of the $\Omega$-limit set of $\varphi_{H|\cl}^{0,1}$ and let
  $(q, p, 0, c-H(q, p, 0))$ be the corresponding point in $\cy$. Then the projection of the $\ch$ orbit of $(q, p,0 , c-H(q, p, 0))$ on $M$ is $(\ug, c , L)$-calibrated.
\end{proposition}
\begin{proof}
  Let us fix $[a, b]\subset \R$ and let us consider the piece of orbit $$t\in[a, b]\mapsto \zeta(t)=\phi_{\ch}^t(q, p, 0, c-H(q, p, 0))=(q(t), p(t), t, c-H(q(t), p(t), t)).$$ Because $(q, p)\in \Omega (\varphi_{H|\cl}^{0,1})$,
  we can find a  sequence of pieces of $\ch$ orbits  $$t\in [a, b_n]\mapsto \zeta_n(t)=(q_n(t),  p_n(t), t,c-H(q_n(t), p_n(t), t))$$
in $\cy$ such that $b_n-a\in \N$,  $\displaystyle{\lim_{n\rightarrow \infty} b_n=+\infty}$, $\displaystyle{\lim_{n\rightarrow \infty} (q_n(a), p_n(a))=(q(a), p(a))}$ and $\displaystyle{\lim_{n\rightarrow \infty} (q_n(b_n), p_n(b_n))=(q(a), p(a))}$.\\
Because of the properties of the defect of calibration $\delta$, we have
$$0\leq\delta (q_{|[a, b]})=\lim_{n\rightarrow \infty} \delta(q_{n|[a, b]})\leq \liminf_{n\rightarrow \infty}\delta (q_{n|[a, b_n]}).$$
We have 
$$\delta (q_{n|[a, b_n]})=\ug(q_n(b_n), b_n)-\ug(q_n(a), a)-\int_a^{b_n}(L(q_n(t), \dot q_n(t), t)+c)dt.
$$
We prove now the following lemma.
\begin{lemma}\label{Lcalcul}
If $t\in\R\mapsto \zeta(t)=(q(t), p(t), t, c-H(q(t), p(t), t))$ is an orbit for $\ch$ on $\cy$, then we have
\begin{equation}\int_a^b(L(q(t), \dot q(t), t)+c)=\cs (\zeta(b))-\cs(\zeta(a)).\end{equation}
\end{lemma}
Because of Proposition \ref{PLiouvilleprim}, we know that   $\cs$ is a primitive of $\theta$ along $\cy$ and so we have 
$$
\cs(\zeta(b))-\cs(\zeta(a))=\int_{\zeta\left|_{[a,b]}\right.} \theta = \int_a^b (p(t).\dot{q}(t)-H(q(t),p(t),t)+c)\, dt
$$
and along every orbit we have $p(t).\dot{q}(t)-H(q(t),p(t),t)= L(q(t), \dot q(t), t)$. This proves the lemma.

Applying Lemma \ref{Lcalcul}, we obtain:
$$\delta (q_{n|[a, b_n]})=\ug(q_n(b_n), b_n)-\ug(q_n(a), a)-(\cs(\zeta_n(b_n))-\cs (\zeta_n(a))).
$$
Using the continuity of $\ug$ and $\cs$ and the fact that $\displaystyle{ \lim_{n\rightarrow \infty}\zeta_n(a)=\zeta(a)}$ and $\displaystyle{ \lim_{n\rightarrow \infty}\zeta_n(b_n)=\zeta(b)}$, we deduce that 
$$\lim_{n\rightarrow \infty}\delta(q_{n|[a, b_n]})=\ug(q(a),a)-\ug(q(a), a)-(\cs (\zeta(a))-\cs(\zeta(a)))=0$$
 and thus $q$ is calibrated.\end{proof}

\subsection{Every orbit in $\cy$ is in the graph of $d\ug$}
\begin{proposition} \label{calibri}
Let $t \mapsto \zeta(t)=\phi_{\ch}^t(q, p, 0, c-H(q, p, 0))=(q(t),  p(t), t, c-H(q(t), p(t), t))$ be an orbit for $\ch$ on $\cy$. Then the curve $q(t)$ is $(\ug, c, L)$-calibrated and we have
$$\forall t\in\R, (p(t), c-H(q(t), p(t), t))=d\ug (q(t), t).$$
\end{proposition}

\begin{proof}
  We choose $(q_+, p_+)\in \omega((q(0), p(0)), \varphi_H^{0,1})$ and
  $(q_-, p_-)\in \alpha((q(0), p(0)), \varphi_H^{0,1})$ and we denote by
  $\zeta_\pm(t)=(q_\pm(t), p_\pm(t), t,c-H(q_\pm(t), p_\pm(t), t)) $ the corresponding $\ch$ orbits in $\cy$. \\
Then there exists two increasing sequences $(n_i)$ and $(m_i)$ of positive integers so that 
\begin{equation}\label{eq00}\lim_{i\rightarrow \infty} \zeta(-m_i)=\zeta_-(0)\quad{\rm and}\quad
  \lim_{i\rightarrow \infty} \zeta(n_i)=\zeta_+(0).\end{equation}
If $[a, b]\subset \R$, we have
$$0\leq \delta(q_{|[a, b]})\leq \liminf_{i\rightarrow \infty}\delta(q_{[-m_i, n_i]}),$$
hence we will prove that this last limit is zero.

Using Lemma \ref{Lcalcul}, we  obtain
$$\begin{matrix}\delta(q_{[-m_i, n_i]})&=\int_{-m_i}^{n_i}(L(q(t), \dot q(t), t)+c)dt-(\ug(q(n_i), n_i)-\ug(q(-m_i), -m_i))
  \hfill\\
&=\cs(\zeta(n_i))-\ug(q(n_i), n_i)-(\cs(\zeta(-m_i))-\ug(q(-m_i), -m_i))\hfill

\end{matrix}$$
Because of (\ref{eq00}) and of the continuity of $\ug$ and $\cs$, we obtain
$$\lim_{i\rightarrow \infty} \delta(q_{[-m_i, n_i]})=\cs(\zeta_+(0))-\ug(q_+(0), 0)-(\cs(\zeta_-(0))-\ug(q_-(0), 0)).$$
We deduce from Proposition \ref{POmega} and Proposition \ref{propcal} that $\ug$ is differentiable at every $(q_\pm(t),t)$  and that
$$d_q\ug(q_\pm(t), t)=\frac{\partial L}{\partial v}(q_\pm(t), \dot q_\pm(t), t)\quad{\rm and}\quad \ch (q_\pm(t),d_q\ug(q_\pm(t), t), t,  \frac{\partial \ug}{\partial t}(q_\pm(t), t))=c.$$
This implies that $d\ug(q_\pm(t), t)=(p_\pm(t), c-H(q_\pm(t), p_\pm(t), t)) $ and then that $\zeta_\pm(t)=(q_\pm(t),d_q\ug(q_\pm(t), t),  t,
\frac{\partial \ug}{\partial t}(q_\pm(t), t))$. This gives that
$$\lim_{i\rightarrow \infty} \delta(q_{[-m_i, n_i]})= 
 \cs(q_+(0), d_q\ug(q_+(0), 0), 0, \frac{\partial \ug}{\partial t}(q_+(0), 0))-\ug(q_+(0), 0)$$
 $$-(\cs(q_-(0), d_q\ug(q_-(0), 0), 0, \frac{\partial \ug}{\partial t}(q_-(0), 0))-\ug(q_-(0), 0)).$$

Proposition  \ref{u-lipschitz-select} tells us that 
$$\forall t\in\R, (q_\pm(t), d_q\ugoth(q_\pm(t), t), t,
\frac{\partial \ugoth}{\partial t}(q_\pm(t), t))\in \cy$$ and$$ \ugoth(q_\pm(t), t)=
\cs(q_\pm(t), d_q\ugoth(q_\pm(t), t), t, \frac{\partial \ugoth}{\partial t}(q_\pm(t), t)).$$
We finally deduce that
$$\lim_{i\rightarrow \infty} \delta(q_{[-m_i, n_i]})=0$$
and that $q$ is $(u, L,c)$-calibrated.

We deduce from Proposition \ref{propcal} that $\ug$ is differentiable at every $(q(t),t)$  and that
$$d_q\ug(q(t), t)=\frac{\partial L}{\partial v}(q(t), \dot q(t), t)\quad{\rm and}\quad \ch (q(t), d_q\ug(q(t), t), t,
\frac{\partial \ug}{\partial t}(q(t), t))=c.$$
This implies that
$\zeta(t)= (q(t), d_q\ug(q(t), t), t, \frac{\partial \ug}{\partial t}(q(t), t))=J_{\ug}(q(t),t)$.
\end{proof}
\subsection{The Lipschitz property of $\cy$}
We know   that $\cy$ is the graph $d\ug$. 
We wish to show   that $d\ug$ is locally Lipschitz. Then we will deduce that $\cl$, which is a $C^1$ Lagrangian submanifold and a locally Lipschitz graph, is the graph of a $C^1$ 1-form.\\

We use Proposition  4.11.3 of \cite{Fa1}:\\

\noindent{\bf Criterion for a Lipschitz derivative (Albert Fathi) .} {\em Let $B=B(x_0, r_0)$ be an open ball in $\R^n$, let $u:B\rightarrow \R$ be a function and let $K>0$ be a positive constant. We introduce the following notation
$$A_{K, u}=\{ x\in B; \exists \varphi_x\in L(\R^n, R), \forall y\in B, |u(y)-u(x)-\varphi_x(y-x)|\leq K\| y-x\|^2\}.$$
Then $u$ has a derivative at every point of $A_{K, u}$ and if $x\in A_{K, u}$, we have $d_xu=\varphi_x$. Moreover, the restriction of $x\mapsto d_xu$ to $\{x\in A_{K, u}; \| x-x_0\|\leq \frac{r_0}{3}\}$ is Lipschitz with a Lipschitz constant equal to $6K$.\\

}
Let us fix $(x_0, t_0)\in M\times [0, 1]$ and let us prove that $d\ug$ is Lipschitz in some neighbourhood of $(x_0, t_0)$. We work in some chart, i.e. $B_0=\bar B(x_0, r_0)$.  Then we choose $B_1=\bar B(x_0, r_1)\subsetneq  B_0$ and
$I_0=[t_0-\varepsilon_0, t_0+\varepsilon_0]$, such that
$$\forall t, \tau\in I_0, \forall x\in B_1, \pi\circ \varphi_H^{t, \tau}(x, d_q\ug(x, t))\in B_0.$$
Let us set $\varepsilon_1=\frac{\varepsilon_0}{4}$ and
$I_1=[t_0-\varepsilon_1,t_0+\varepsilon_1]$.
If $(x, t)\in B_1\times I_1$, we consider the $(\ug, L, c)$-calibrated curve
$$\begin{matrix}\Gamma=\Gamma_{x, t}: &[t-3\varepsilon_1, t+3\varepsilon_1]&\rightarrow& B_0\hfill\\
&\hfill s&\mapsto& \pi\circ \varphi_H^{t,s} (x, d_q\ug)(x, t)).\end{matrix}$$
Then, for every $(y, \tau)\in B_1\times I_1$, we can build two curves (as we did in the proof of Proposition \ref{propcal})
\begin{itemize}
\item$\gamma_{y, \tau}=\gamma_{y, \tau, x, t}: [t-3\varepsilon_1, \tau]\rightarrow B_0$ defined by 
$$\gamma_{y, \tau}(s)=\Gamma_{x,t}(s)+\frac{s-(t-3\varepsilon_1)}{\tau-(t-3\varepsilon_1)}(y-\Gamma_{x,t}(\tau));
$$
\item $\eta_{y, \tau}=\eta_{y, \tau, x, t}:[\tau, t+3\varepsilon_1]\rightarrow B_1$ defined by 
$$\eta_{y, \tau}(s)=\Gamma_{x,t}(s)+\frac{(t+3\varepsilon_1)-s}{(t+3\varepsilon_1)-\tau}(y-\Gamma_{x,t}(\tau)).
$$
\end{itemize}
Observe that $\gamma_{x, t}=\Gamma_{x,t}\left|_{[t-3\varepsilon_1, \tau]}\right.$ and
$\eta_{x,t}=\Gamma_{x,t}\left|_{[\tau, t+3\varepsilon_1]}\right.$.
We then define
$$\psi_+(y, \tau)=\psi_{+,x, t}(y, \tau)=\ug(\Gamma_{x,t}(t-3\varepsilon_1), t-3\varepsilon_1)
+\int_{t-3\varepsilon_1}^\tau \left(L(\gamma_{y, \tau}(s), \dot \gamma_{y, \tau}(s), s)+c\right)ds$$
and
$$\psi_-(y, \tau)=\psi_{-,x, t}(\y, \tau)=\ug(\Gamma_{x,t}(t+3\varepsilon_1), t+3\varepsilon_1)
-\int_\tau ^{t+3\varepsilon_1}\left(L(\eta_{y, \tau}(s), \dot \eta_{y, \tau}(s), s)+c\right)ds.$$
Let us recall that $\psi_-\leq \ug\leq \psi_+$, that $\psi_-$ and $\psi_+$ are $C^1$ and that $\psi_-(x, t)=\ug(x, t)=\psi_+(x,t)$.\
In particular,
\begin{equation}\label{Ineg1}\ug(y, \tau)-\ug(x, t)\leq \psi_+(y, \tau)-\psi_+(x, t)=
  \alpha+c(\tau-t)+\beta\end{equation}
where 
$$\alpha=\int_{t-3\varepsilon_1}^{t}\left( L(\gamma_{y, \tau}(s), \dot\gamma_{y, \tau}(s), s)- L(\gamma_{x,t}(s), \dot\gamma_{x,t}(s), s)\right)ds$$
and 
$$\beta=\int_{t}^{\tau} L(\gamma_{y, \tau}(s), \dot\gamma_{y, \tau}(s), s)ds.$$
Because each curves $s\mapsto \frac{\partial L}{\partial v}(\Gamma_{x, t}(s), \dot \Gamma_{x, t}(s), s)$ is drawn on $\cy$ and then bounded, there exists a constant $K_0$ such that
$$\forall (x, t)\in B_1\times I_1, \ \forall s\in [t-3\varepsilon_1, t+3\varepsilon_1], \
\|Ê\dot\Gamma_{x,t}(s)\|Ê\leq K_0.$$
Because $L$ is $C^2$, there also exists a constant $K_1\geq K_0$ such the first and second derivative of $L$ are bounded by $K_1$ on the set
$$B_1\times B(0, K_0+\frac{2r_0}{\varepsilon_1})\times I_1.$$
\begin{remk}
Because $\dot \gamma_{y, \tau}(s)=\dot\gamma_{x,t}(s)+\frac{y-\Gamma_{x,t}(\tau)}{\tau-(t-3\varepsilon_1)}$, we have
\begin{equation}\label{Supder} \|\dot \gamma_{y, \tau}(s)\|\leq K_0+\frac{1}{\varepsilon_1}\| y-\Gamma_{x,t}(\tau)\|\leq K_0+\frac{2r_0}{\varepsilon_1}=K_2.\end{equation}
\end{remk}
In the following of the proof we denote by $K_3$, $K_4$, $K_5$ ... some positive constants depending
only
on the restriction of $L$ to the set $B_0\times B(0,K_2)\times I_0$. 
 Taylor-Lagrange inequality implies that for every $s\in [t-3\varepsilon_1, \tau]$, we have
 $$\big|L(\gamma_{y, \tau}(s), \dot\gamma_{y, \tau}(s), s)-L(\gamma_{x,t}(s), \dot\gamma_{x,t}(s), s)-\frac{\partial L}{\partial x}(\gamma_{x,t}(s), \dot\gamma_{x,t}(s), s)
 \frac{s-(t-3\varepsilon_1)}{\tau-(t-3\varepsilon_1)}(y-\gamma_{x,t}(\tau))$$
\begin{equation}\label{Efin}\hfill-\frac{\partial L}{\partial v}(\gamma_{x,t}(s), \dot\gamma_{x,t}(s), s)\frac{y-\gamma_{x,t}(\tau)}{\tau-(t-3\varepsilon_1)}\big| 
\leq \frac{K_3}{\varepsilon_1^2}\| y-\gamma_{x,t}(\tau)\|^2.
\end{equation} 
Using Euler-Lagrange equations, inequality  (\ref{Efin}) and an integration by parts, since
$\gamma_{x,t}=\Gamma_{x,t}\left|_{[t-3\epsilon_1,\tau]}\right.$, we get
the following inequality
$$\left| \alpha-\left[ \frac{\partial L}{\partial v}(\Gamma_{x,t}(s), \dot\Gamma_{x,t}(s), s)
  \frac{s-(t-3\varepsilon_1)}{\tau-(t-3\varepsilon_1)}(y-\Gamma_{x,t}(\tau))\right]_{s=t-3\varepsilon_1}^{s=t}\right| \leq \frac{3K_3}{\varepsilon_1}\| y-\Gamma_{x,t}(\tau)\|^2$$
i.e.
$$\left| \alpha- \frac{3\varepsilon_1}{\tau-(t-3\varepsilon_1)}\frac{\partial L}{\partial v}(\Gamma_{x,t}(t), \dot\Gamma_{x,t}(t),t)(y-\Gamma_{x,t}(\tau)) \right| \leq \frac{3K_3}{\varepsilon_1}\| y-\Gamma_{x,t}(\tau)\|^2.$$
We deduce from inequality (\ref{Supder}) that
$\|\Gamma_{x,t}(\tau)-\Gamma_{x,t}(t)\|\leq K_2|t-\tau|$ and then $\|y-\Gamma_{x,t}(\tau)\|=\|(y-x)+(x-\Gamma_{x,t}(\tau))\|\leq \|y-x\|+K_2|t-\tau|$. We note too that
$|\frac{3\varepsilon_1}{\tau-(t-3\varepsilon_1)}-1|\leq \frac{|t-\tau|}{\varepsilon_1}$ and so
$$\left| \alpha- d_q\ug (x, t)(y-\Gamma_{x,t}(\tau)) \right| \leq K_4(\|  y-x\|+|t-\tau|)^2.$$
Observe that Euler-Lagrange Equation implies that the $\ddot{\Gamma}_{x,t}$s are uniformly bounded by some constant. Hence 
$$\|y-\Gamma_{x,t}(\tau)-(y-x-\dot{\Gamma}_{x,t}(t)(\tau-t))\|\leq K_5(\tau -t)^2.$$
We deduce
\begin{equation}\label{Ineg2}\left| \alpha- d_q\ug (x, t)(y-x-\dot{\Gamma}_{x,t}(t)(\tau-t))\right|
  \leq K_6(\|  y-x\|+|t-\tau|)^2.\end{equation}
In a similar way, we obtain
\begin{equation}\label{Ineg3}\left|\beta-(\tau-t)L(x, \dot{\Gamma}_{x,t}(t), t)\right|\leq K_7(\|  y-x\|+|t-\tau|)^2.\end{equation}
 Equations (\ref{Ineg1}),  (\ref{Ineg2}) and (\ref{Ineg3}) imply that
$$ \ug(y, \tau)-\ug(x, t)-d_q\ug(x, t)(y-x)-(\tau-t)\frac{\partial \ug}{\partial t}(x, t)\leq  K_8(\|  y-x\|+|t-\tau|)^2.$$
Using $\psi_-$ instead of $\psi_+$, we obtain then 
$$| \ug(y, \tau)-\ug(x, t)-d_q\ug(x, t)(y-x)-(\tau-t)\frac{\partial \ug}{\partial t}(x, t)|\leq  K_9(\|  y-x\|+|t-\tau|)^2.$$
Using the criterion for a Lipschitz derivative, we conclude.

\subsection{Proof of the corollaries}
We prove the corollaries that were given in the introduction.\\

{\bf Proof or Corollary \ref{corauton}} With the hypothesis of the corollary, we obtain that $\cl$ is a graph. We then use Theorem 6.4.1 of \cite{Fa1}, which is a corollary of the convergence of the Lax-Oleinik semi-group in weak KAM theory, to conclude.

{\bf Proof or Corollary \ref{cormin}} We proved that the $\ch$ orbit of every point in $\cy$ is $(u, L,c)$-calibrated. This implies (see for example \cite{Be1}) that every orbit is minimizing.
\vspace{2mm}\\ \noindent
       {\bf Acknowledgements.} The authors thank sincerely the  referee for   many valuable suggestions and comments, who greatly improve the the
       quality of the paper.

%
%
%
%

\bibliographystyle{amsplain}

\end{document}